\documentclass[a4paper,11pt]{amsart}
\usepackage[latin1]{inputenc}
\usepackage{graphicx}
\usepackage{amsmath}
\usepackage{amsfonts}
\usepackage{amssymb}
\usepackage{amsthm}
\usepackage{hyperref}
\usepackage[active]{srcltx}
\usepackage{bbm}

\newcommand{\R}{{\mathbb{R}}}
\newcommand{\Q}{{\mathbb{Q}}}
\newcommand{\Z}{{\mathbb{Z}}}
\newcommand{\C}{\mathbb{C}}

\newcommand{\PP}{\mathbf{P}}
\newcommand{\EE}{\mathbb{E}}

\def\vecx{{\text{\boldmath$x$}}}

\def\vecu{{\text{\boldmath$u$}}}
\def\vecv{{\text{\boldmath$v$}}}
\def\wv{\widehat{\text{\boldmath$v$}}}

\def\vec0{{\text{\boldmath$0$}}}

\def\scrP{{\mathcal P}}
\def\scrV{{\mathcal V}}

\def\fS{{\mathfrak S}}
\def\one{{\mathbbm{1}}}

\def\sgn{\operatorname{sgn}}
\def\Res{\operatorname{Res}}
\def\supp{\operatorname{supp}}
\def\L{\operatorname{L{}}}

\newcommand{\col}{\: : \:}

\newcommand{\oJ}{\overline{J}}
\newcommand{\ve}{\varepsilon}

\newcommand{\sfrac}[2]{{\textstyle \frac {#1}{#2}}}

\newcommand{\SL}{\mathrm{SL}}

\newtheorem{thm}{Theorem}%
\newtheorem{lem}{Lemma}
\newtheorem{prop}{Proposition}
\newtheorem{cor}{Corollary}

\theoremstyle{remark}
\newtheorem{remark}{Remark}

\numberwithin{equation}{section}

\begin{document}
\title[On the zero-free half-plane of a random Epstein zeta function]{On the location of the zero-free half-plane of a random Epstein zeta function}
\author{Andreas Str\"ombergsson and Anders S\"odergren}
\address{Department of Mathematics, Box 480, Uppsala University, 751 06 Uppsala, Sweden\newline
\rule[0ex]{0ex}{0ex}\hspace{8pt} {\tt astrombe@math.uu.se}\newline
\newline
\rule[0ex]{0ex}{0ex} \hspace{8pt}School of Mathematics, Institute for Advanced Study, Einstein Drive, Princeton,\newline
\rule[0ex]{0ex}{0ex}\hspace{8pt} NJ 08540, USA\newline
\rule[0ex]{0ex}{0ex} \hspace{8pt}{\tt sodergren@math.ias.edu}}  
\date{\today}
\thanks{The first author is supported by a grant from the G\"oran Gustafsson Foundation for
Research in Natural Sciences and Medicine, and also by the Swedish Research Council Grant 621-2011-3629. The second author was funded by a postdoctoral fellowship from the Swedish Research Council. This material is based upon work supported in part by the National Science Foundation under agreement No.\ DMS-0635607. %
 }

\begin{abstract}
In this note we study, for a random lattice $L$ of large dimension $n$, the supremum of the real parts of the zeros of the Epstein zeta function $E_n(L,s)$ and prove that this random variable scaled by $n^{-1}$ has a limit distribution, which we give explicitly. This limit distribution is studied in some detail; in particular we give an explicit formula for its distribution function. Furthermore, we obtain a limit distribution for the frequency of zeros of $E_n(L,s)$ in vertical strips contained in the half-plane $\Re s>\frac n2$.
\end{abstract}

\maketitle

\section{Introduction}

Let $X_n$ denote the space of all $n$-dimensional lattices $L\subset\R^n$ of covolume one.
For $L\in X_n$ and $\Re s>\frac{n}{2}$, the Epstein zeta function is defined by
\begin{align}\label{epsteindefinition}
E_n(L,s)={\sum_{\vecv\in L}}'|\vecv|^{-2s},
\end{align}
where $'$ denotes that the zero vector should be omitted. $E_n(L,s)$ has an analytic continuation to $\C$ except for a simple pole at $s=\frac{n}{2}$ with residue $\pi^{\frac n2}\Gamma(\frac{n}{2})^{-1}$. Furthermore, $E_n(L,s)$ satisfies the functional equation
\begin{align}\label{FCEQU}
F_n(L,s)=F_n(L^*,\sfrac{n}{2}-s),
\end{align}
where $F_n(L,s):=\pi^{-s}\Gamma(s)E_n(L,s)$ and $L^*$ is the dual lattice of $L$.

The Epstein zeta function is in many ways analogous to the Riemann zeta function. In particular we have the relation
\begin{align*}
E_1(\Z,s)=2\zeta(2s).
\end{align*}
Because of this analogy and for other related reasons, many studies have been made regarding 
the location of the zeros of $E_n(L,s)$.
From \eqref{FCEQU} it is clear that $E_n(L,s)$ has a ``trivial'' zero at each point $s=-1,-2,-3,\ldots$,
just like $\zeta(2s)$, and the remaining nontrivial zeros of $E_n(L,s)$ are in bijective correspondence with
the nontrivial zeros of $E_n(L^*,s)$ under the map $s\mapsto \frac n2-s$.
However, the Riemann hypothesis for $E_n(L,s)$ generally fails:
$E_n(L,s)$ typically has many nontrivial zeros which do not lie on the critical line $\Re s=\frac n4$.
Cf.\ \cite{DH}, \cite{bateman}, \cite{stark}, \cite{terras1}, \cite{terras2}, \cite{terras3}, \cite{steuding}.

We denote by $N_L(T)$ the number of nontrivial zeros (counting multiplicity) of $E_n(L,s)$ with
$|\Im s|\leq T$. Then $N_L(T)$ satisfies the following Riemann-\mbox{von Mangoldt} type asymptotics (\cite{steuding}):
\begin{align}\label{STEUDINGRVMFORMULA}
N_L(T)=\frac{2T}\pi\log\frac T{\pi e\, m(L) m(L^*)}+O_L(\log T) \qquad\text{as }\: T\to\infty,
\end{align}
where $m(L)$ is the length of the shortest non-zero vector in $L$.

From the point of view of number theory, the most interesting choices of $L$ are those for which
the Gram matrix for some (and thus any) $\Z$-basis of $L$ is proportional to an integer matrix.
We call these lattices \textit{rational.}
In particular when $n=2$ many results have been obtained regarding the zeros of $E_2(L,s)$ for rational $L$ corresponding to integral quadratic forms with a fundamental discriminant.
It was conjectured by H.\ L.\ Montgomery that in this case asymptotically $100\%$ of the nontrivial zeros of $E_2(L,s)$
lie along the critical line $\Re s=\frac12$.
This was proved conditionally,
assuming the Generalized Riemann Hypothesis and a weak form of well-spacing for the zeros of 
$L$-functions attached to ideal class characters,
by Bombieri and Hejhal in \cite{BH2}.
Furthermore, Selberg has proved unconditionally, in still unpublished work
(cf.\ \cite[p.\ 553]{hejhal2000} and \cite[pp.\ 225-227]{BG})
that a positive proportion of the zeros do lie on the critical line.
For related results, see also \cite{JS} and \cite{MRS}.

\vspace{5pt}

Our main object of study in the present paper is the supremum of the real parts of the zeros of $E_n(L,s)$, i.e.
\begin{align*}
\sigma_L:=\sup\bigl\{\Re\rho\col E_n(L,\rho)=0\bigr\}.
\end{align*}
In other words, 
$\sigma_L$ gives the precise location of the zero-free right half-plane of $E_n(L,s)$.
One easily shows that $\sigma_L$ exists and is finite %
for any given $L\in X_n$; furthermore $\sigma_L\geq\frac n4$ always holds
(cf., e.g., \cite[p.\ 693 and Thm.\ 1]{steuding}).
Of course, $\sigma_L=\sigma_{L^*}=\frac n4$ %
is equivalent with the Riemann hypothesis for $E_n(L,s)$. 
Note that $\sigma_L$ is lower semicontinuous (and hence Borel measurable),
since %
any zero $s=s_0$ of $E_n(L_0,s)$ gives rise to a nearby zero for all $E_n(L,s)$ with $L$ in a sufficiently small neighborhood of $L_0$
(as follows from a standard application of Rouche's theorem
using the formula \cite[(23)]{SS} for $\pi^{-s}\Gamma(s)E_n(L,s)$). 
We also remark that %
$\sigma_L$ takes arbitrarily large values for any given $n\geq2$
(cf.\ Remark \ref{TILDESIGMALSMOOTHREMARK} in Section \ref{BASICPROPPROOFSEC}).

For any Dirichlet series $f(s)=\sum_{j=1}^\infty e^{-\lambda_j s}$ with exponents
$\lambda_1<\lambda_2<\ldots$ whose pairwise differences do not satisfy any non-trivial linear relation over $\Q$,
the supremum of the real parts of the zeros of $f(s)$ equals the unique
number $\sigma$ for which $e^{-\lambda_1\sigma}=\sum_{j=2}^\infty e^{-\lambda_j\sigma}$;
cf.\ Lemma \ref{BASICLEM1} below. %
This independence condition never holds for $E_n(L,s)$ (e.g.\ since $L$ contains both $2\vecv$ and $4\vecv$ for any 
$\vecv\in L$).
However, we have
\begin{align}\label{primitiverelation2}
E_n(L,s)=2\zeta(2s)\sum_{\vecv\in \widehat L}|\vecv|^{-2s}\qquad (\Re s>\tfrac n2),
\end{align} 
where $\widehat L$ denotes a set containing one representative from each pair $\{\vecv,-\vecv\}$
of \textit{primitive} vectors in $L$;
and it turns out that the Dirichlet series $\sum_{\vecv\in \widehat L}|\vecv|^{-2s}$
satisfies the independence condition
for \textit{$\mu_n$-almost every} lattice $L\in X_n$,
where $\mu_n$ is Siegel's measure (\cite{siegel}) on $X_n$;
see Lemma \ref{BASICLEM2}.
From this we conclude (cf.\ Section \ref{BASICPROPPROOFSEC}):

\begin{prop}\label{BASICPROP}
Let $n\geq2$.
For almost every $L\in X_n$, $\sigma_L$ equals the unique number $\sigma>\frac n2$ which satisfies
$2m(L)^{-2\sigma}  %
=\frac12\zeta(2\sigma)^{-1}E_n(L,\sigma)$.
It follows that for almost every $L\in X_n$, $E_n(L,s)$ has infinitely many zeros with $\Re s>\frac n2$.
\end{prop}
In particular, for small $n$ the formula in Proposition \ref{BASICPROP} makes it
possible to compute $\sigma_L$ numerically for a given generic $L\in X_n$.
We stress, however, that for a lattice $L$ such that
$\sum_{\vecv\in \widehat L}|\vecv|^{-2s}$ does \textit{not} satisfy the linear independence condition
(e.g.\ any rational $L$, cf.\ Remark \ref{RATIONALNONGENERICREMARK} in Section \ref{BASICPROPPROOFSEC}),
the computation of $\sigma_L$ is in general not an easy task.
We mention that Bombieri and Mueller in \cite{BM} have shown how to calculate $\sigma_L$ explicitly for
certain examples of rational lattices $L\in X_2$ (with $\sigma_L>1$),
where they also obtained bounds 
on the asymptotic rate of approach of the zeros of $E_2(L,s)$ to the line $\Re s=\sigma_L$. 
See also \cite{BG} for a related investigation of the supremum of the real parts of the zeros of certain other Dirichlet series.

Our main result concerns the distribution of $\sigma_L$ for a \textit{random} lattice $L$ in large dimension $n$.
The random element $L\in X_n$ will always be chosen according to Siegel's measure $\mu_n$, normalized to be a probability measure.
The present study is motivated by recent investigations \cite{epstein1} of the value distribution of $E_n(L,s)$ for  $\Re s>\frac{n}{2}$ and a $\mu_n$-random lattice $L$ of large dimension $n$, where the following result is established: Let $V_n$ denote the volume of the $n$-dimensional unit ball. Let $\mathcal P$ be a Poisson process on the positive real line with intensity $\frac{1}{2}$ and let $T_1,T_2,T_3,\ldots$ denote the points of $\mathcal P$ ordered so that $0<T_1< T_2< T_3<\cdots$. Then, for any fixed $s\in\C$ with $\Re s>\frac{1}{2}$,
\begin{align}\label{EPSTEINMAINRES}
V_n^{-2s}E_n(\cdot,ns)\xrightarrow[]{\textup{ d }}
2\sum_{j=1}^\infty T_j^{-2s}
\qquad\text{as }\: n\to\infty,
\end{align} 
i.e.\ the random variable $V_n^{-2s}E_n(\cdot,ns)$ converges in distribution to $2\sum_{j=1}^\infty T_j^{-2s}$.

The proof of \eqref{EPSTEINMAINRES} is built on a result \cite{poisson} which provides the connection between the lengths of lattice vectors appearing in the formula \eqref{epsteindefinition} and the points of the Poisson process $\mathcal P$. Since this result is an important ingredient also in the present investigation we recall it here. Given a lattice $L\in X_n$, we order its non-zero vectors by increasing lengths as $\pm\vecv_1,\pm\vecv_2,\pm\vecv_3,\ldots$, set $\ell_j=|\vecv_j|$ 
(thus $0<\ell_1\leq \ell_2\leq\ldots$), and define
\begin{align}\label{volumes}
 \mathcal V_j(L):=%
 V_n\ell_j^n\,,
\end{align}
so that $\mathcal V_j(L)$ is the volume of an $n$-dimensional ball of radius $\ell_j$. The main result in \cite{poisson} states that, as $n\to\infty$, the volumes $\{\mathcal V_j(L)\}_{j=1}^{\infty}$ determined by a random lattice $L\in X_n$ converges in distribution to the points $\{T_j\}_{j=1}^{\infty}$ of the  Poisson process $\mathcal P$ on the positive real line with constant intensity $\frac{1}{2}$.

In view of the last two paragraphs, together with %
Proposition \ref{BASICPROP} and the fact that
$\zeta(2\sigma)\to1$ as $\sigma\to\infty$, it seems reasonable to expect that
as $n\to\infty$, $n^{-1}\sigma_L$ should tend in distribution to %
\begin{align}\label{SIGMATJDEF}
\sigma_{\{T_j\}}:=\Bigl[\text{the unique $\sigma>\tfrac12$ satisfying $T_1^{-2\sigma}=\sum_{j=2}^\infty T_j^{-2\sigma}$}\Bigr].
\end{align}
(We will show in Section \ref{MAINRESULT1PROOFSEC} that $\sigma_{\{T_j\}}$ is a well-defined random variable.)
Our first theorem states that this is indeed the case.
\begin{thm}\label{mainresult}
If $L$ is taken at random in $X_n$ according to $\mu_n$, then
\begin{align*}
n^{-1}\sigma_L\xrightarrow[]{\textup{ d }}\sigma_{\{T_j\}}\qquad\text{as }\: n\to\infty.
\end{align*}
\end{thm}

\vspace{5pt}

By similar techniques we also obtain a limit distribution statement concerning the
\textit{frequency} of zeros of $E_n(L,s)$ in arbitrary vertical strips to the right of $\Re s=\frac n2$.
For any %
$\sigma_1<\sigma_2$ and $\tau_1<\tau_2$,
let $N_L(\sigma_1,\sigma_2;\tau_1,\tau_2)$ be the number of zeros of $E_n(L,s)$ in the rectangle
$s\in(\sigma_1,\sigma_2)\times(\tau_1,\tau_2)$, counting multiplicity.
It follows from Jessen \cite[Satz A]{Jessen1933} that for each $L\in X_n$, 
and for any fixed numbers $\sigma_1,\sigma_2\in(\frac n2,\infty)\setminus\fS_L$, $\sigma_1<\sigma_2$,
where $\fS_L$ is a certain finite or countable set of exceptions,
the limit
\begin{align}\label{HLDEF}
H_L(\sigma_1,\sigma_2):=\lim_{\tau_2-\tau_1\to\infty}\frac{N_L(\sigma_1,\sigma_2;\tau_1,\tau_2)}{\tau_2-\tau_1}
\end{align}
exists.
For $L$ satisfying the independence condition discussed above
(recall that this holds for $\mu_n$-almost every $L\in X_n$), $\fS_L$ is empty, i.e.\
the limit \eqref{HLDEF} exists for \textit{all} $\frac n2<\sigma_1<\sigma_2$, and 
we furthermore have
\begin{align}\label{HLREL}
H_L(\sigma_1,\sigma_2)=\int_{\sigma_1}^{\sigma_2}\nu_L(\sigma)\,d\sigma,
\end{align}
where $\nu_L(\sigma)$ is a continuous function on $(\frac n2,\infty)$.
These statements follow from Jessen's work %
\cite{Jessen1934} 
(cf.\ Section \ref{ZERODENSITYLIMITTHMSEC} below).
We remark that in recent work by Lee \cite{lee} and Gonek-Lee \cite{GonekLee}, 
similar asymptotics for the number of zeros of $E_2(L,s)$ are obtained in the more difficult case of a rational $L\in X_2$ 
corresponding to an integral quadratic form with a fundamental discriminant.

Similarly, for almost any realization of the Poisson process $\scrP$,
there is a continuous function $\nu_{\{T_j\}}$ on $(\frac12,\infty)$ such that,
if $N(\sigma_1,\sigma_2;\tau_1,\tau_2)$ denotes
the number of zeros of the Dirichlet series $f_{\{T_j\}}(s):=\sum_{j=1}^\infty T_j^{-2s}$ in the rectangle
$s\in(\sigma_1,\sigma_2)\times(\tau_1,\tau_2)$,
then for any $\frac12<\sigma_1<\sigma_2$, %
\begin{align}\label{NUTJdef}
\lim_{\tau_2-\tau_1\to\infty}\frac{N(\sigma_1,\sigma_2;\tau_1,\tau_2)}{\tau_2-\tau_1}
=\int_{\sigma_1}^{\sigma_2}\nu_{\{T_j\}}(\sigma)\,d\sigma.
\end{align}
Let $C(\frac12,\infty)$ be the set of all real-valued continuous functions on $(\frac12,\infty)$,
provided with the topology of uniform convergence on compacta.

\begin{thm}\label{ZERODENSITYLIMITTHM}
If $L$ is taken at random in $X_n$ according to $\mu_n$, then
\begin{align*}
\bigl(n^2\nu_L(n\,\cdot\,)   %
,n^{-1}\sigma_L\bigr)\xrightarrow[]{\textup{ d }}(\nu_{\{T_j\}},\sigma_{\{T_j\}})
\qquad\text{as }\: n\to\infty,
\end{align*}
in the sense of 
convergence in distribution for random elements in $C(\frac12,\infty)\times\R_{>1/2}$.
\end{thm}

Note that Theorem \ref{ZERODENSITYLIMITTHM} generalizes Theorem \ref{mainresult},
and also implies that the random function $\sigma\mapsto n^2\nu_L(n\sigma)$ %
in $C(\frac12,\infty)$ converges in distribution to $\nu_{\{T_j\}}$.
A consequence of the latter fact is that for any fixed $\frac12<\sigma_1<\sigma_2$,
the real-valued random variable $nH_L(n\sigma_1,n\sigma_2)$ converges in distribution to
$\int_{\sigma_1}^{\sigma_2}\nu_{\{T_j\}}(\sigma)\,d\sigma$.

\vspace{5pt}

Returning to our main objects of study, i.e.\ $\sigma_L$ and its limit $\sigma_{\{T_j\}}$,
we next give  %
an explicit formula for the distribution function of $\sigma_{\{T_j\}}$.
Recall that the lower incomplete gamma function $\gamma(s,z)$ is defined by
\begin{align}\label{INCOMPLETEGAMMADEF}
\gamma(s,z):=\int_0^z u^{s-1}e^{-u}\,du
\end{align}
for $s,z\in\C$ with $\Re s>0$.
In order to make $\gamma(s,z)$ single-valued, we will always keep $z\in\C\setminus\R_{<0}$
(in fact, we will only need to use $z$ with $\Re z\geq0$),
and choose a path of integration in \eqref{INCOMPLETEGAMMADEF} which stays inside this cut plane.
We agree that $|\arg u|<\pi$ for all $u\in\C\setminus\R_{\leq0}$.
Now the function $\gamma(s,z)$ is extended to all $s\in\C\setminus\Z_{\leq0}$, $z\in\C\setminus\R_{<0}$ 
through the recursion formula
\begin{align}\label{INCOMPLETEGAMMAREC}
\gamma(s,z)=\frac{\gamma(s+1,z)+z^se^{-z}}s.
\end{align}

\begin{thm}\label{EXPLFORMULATHM}
For any $c>\frac12$, we have
\begin{align*}
Prob\bigl(\sigma_{\{T_j\}}\leq c\bigr)=\frac12+\frac{2c}{\pi}\int_0^\infty
\Im\left(\frac{e^{(\frac\pi{4c}-y)i}}{\gamma(-\frac1{2c},-iy)}\right)y^{-1-\frac1{2c}}\,dy.
\end{align*}
The integral in the right-hand side is absolutely convergent.
\end{thm}

\begin{cor}\label{EXPLFORMULACOR}
The random variable $\sigma_{\{T_j\}}$ has a continuous density function $f(c)$ %
given explicitly in \eqref{DENSITY} below. It satisfies  %
\begin{align}\label{ASYMPTCSMALLRES}
f(c)=2-K_1(c-\tfrac12)^2+O((c-\tfrac12)^3)\qquad\text{as }\: c\to\tfrac12
\end{align}
and
\begin{align}\label{ASYMPTCLARGERES}
f(c)=K_2c^{-3}+O\left(c^{-4}\right)\qquad\text{as }\: c\to\infty,
\end{align}
where $K_1=39.47841\ldots$ and $K_2=0.822467\ldots$
are positive real numbers given explicitly below in
\eqref{K1DEF} and  \eqref{KINT}, respectively
(cf.\ also \eqref{FKYDEF}, \eqref{F1F2FORMULAS}, \eqref{MYPIDEF} and \eqref{HFORMULA}).
\end{cor}

\begin{figure}\label{GRAPH}
\begin{center}
\begin{minipage}{0.5\textwidth}
\unitlength0.1\textwidth
\begin{picture}(10,9)(0,0)
\put(-2,9){\includegraphics[width=0.9\textwidth,angle=270]{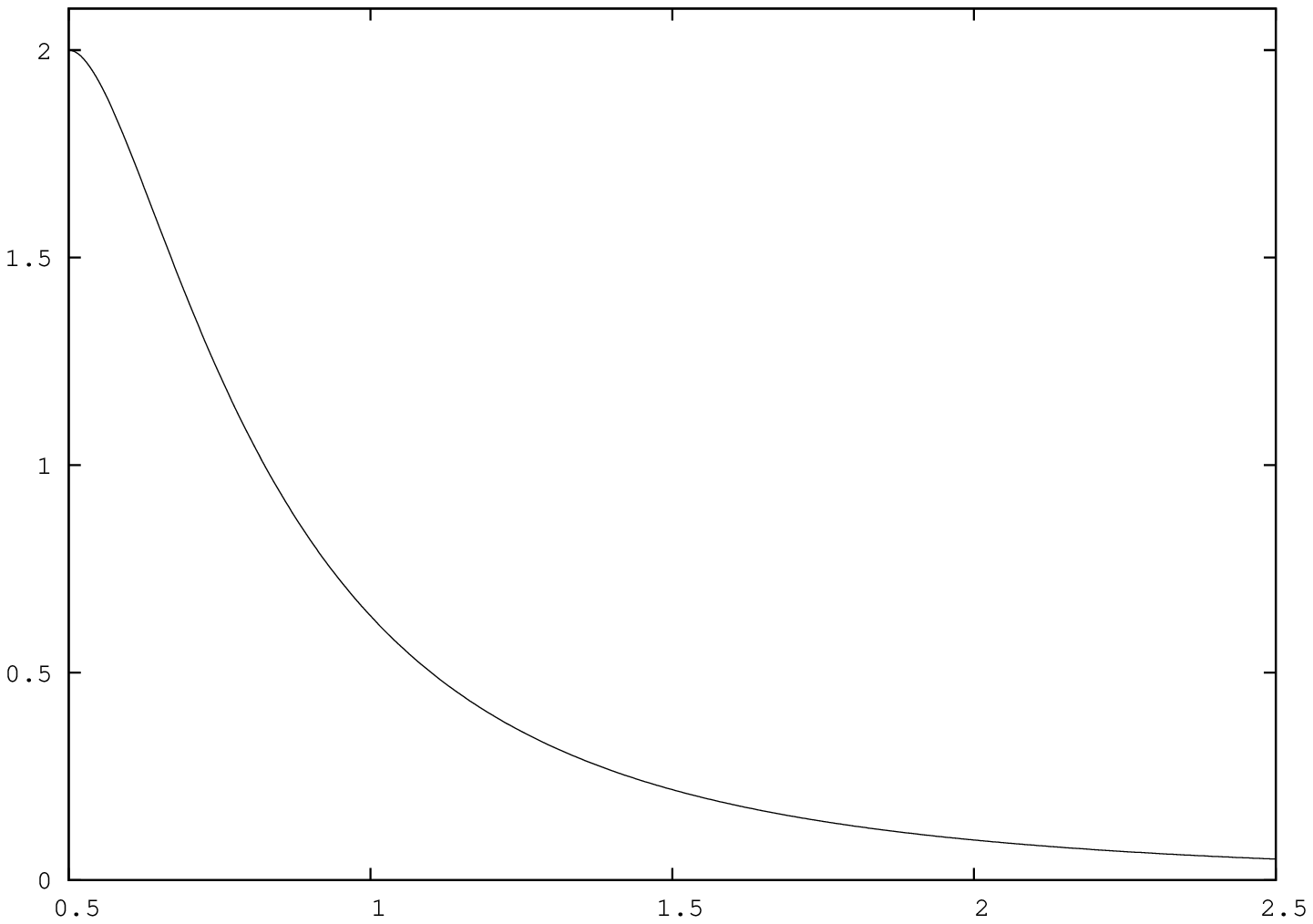}}
\end{picture}
\end{minipage}
\end{center}
\caption{A graph of the density function of $\sigma_{\{T_j\}}$.
It was computed using the method described in Appendix \ref{NUMSEC} (see also \cite[numdensity.mpl]{sts}).}\label{DENSITYGRAPH}
\end{figure}

In Appendix \ref{NUMSEC}, we also give formulas for the distribution and density functions of $\sigma_{\{T_j\}}$
obtained through the residue theorem, and discuss numerical evaluation.
See Figure 1 for a graph of the probability density function of $\sigma_{\{T_j\}}$ generated using the formulas in Appendix \ref{NUMSEC} (cf. \cite[numdensity.mpl]{sts}).

\vspace{5pt}

We conclude by remarking that, as is rather clear from the previous discussion,
the random variable $\sigma_{\{T_j\}}$ can also be interpreted as the
supremum of the real parts of the zeros of the \textit{random Dirichlet series} $f_{\{T_j\}}(s)=\sum_{j=1}^\infty T_j^{-2s}$.
Indeed, by the strong law of large numbers the series $f_{\{T_j\}}(s)$ has, with probability one, abscissa of absolute convergence $\sigma_0=\frac12$ and satisfies %
$\lim_{\sigma\to\frac12+}f_{\{T_j\}}(\sigma)>2T_1^{-1}$;
also with probability one the numbers $2(\log T_j-\log T_1)$, $j=2,3,\ldots$, are linearly independent over $\Q$;
this means that Lemma \ref{BASICLEM1} below applies almost surely and the claim follows. 
Hence Theorem \ref{EXPLFORMULATHM} and Corollary \ref{EXPLFORMULACOR} 
describe explicitly the distribution of the location of the zero-free right half-plane
of $f_{\{T_j\}}(s)$. %
It would be interesting to also seek a more explicit understanding of the random function
$\nu_{\{T_j\}}$ (cf.\ \eqref{NUTJdef} and Theorem~\ref{ZERODENSITYLIMITTHM}), 
which describes the density of zeros of $f_{\{T_j\}}(s)$ in any vertical strip to the right of $\Re s=\frac12$.

There exists a vast literature on random Dirichlet series; however, we are not aware of many results
pertaining to their zeros.
Cf., however, Edelman and Kostlan \cite[\S\S 3.2.5, 8.2]{EK}, regarding the zeros of the random Dirichlet series
$\sum_{n=1}^\infty a_n n^{-s}$, where $a_n$ are independent standard normal random variables.

\subsection{Acknowledgments}
We are grateful to Daniel Fiorilli and Svante Janson for inspiring discussions and helpful remarks. 
We are also grateful to the referee for asking about the density of zeros;
this inspired us to add Theorem \ref{ZERODENSITYLIMITTHM} to the paper.
The second author thanks the Institute for Advanced Study for providing excellent working conditions.

\section{Proof of Proposition \ref{BASICPROP}}
\label{BASICPROPPROOFSEC}

\begin{lem}\label{BASICLEM1}
Consider any Dirichlet series $f(s)=\sum_{j=1}^\infty e^{-\lambda_j s}$ with real exponents
$\lambda_1<\lambda_2<\ldots$ and abscissa of absolute convergence $\sigma_0<\infty$.
Assume $\lim_{\sigma\to\sigma_0^+}f(\sigma)>2e^{-\lambda_1\sigma_0}$.
Then the equation $f(\sigma)=2e^{-\lambda_1\sigma}$ has exactly one real root $\sigma=\sigma_f>\sigma_0$.
If furthermore all the differences $\lambda_j-\lambda_1$ for $j=2,3,\ldots$ are linearly independent over $\Q$,
then $\sigma_f$ equals the supremum of the real parts of the zeros of $f(s)$,
and the function $f(s)$ has infinitely many zeros in any strip $\sigma_1<\Re s<\sigma_2$
with $\sigma_0\leq\sigma_1<\sigma_2\leq\sigma_f$.
\end{lem}
\begin{remark}\label{INDEPENDENCEEQUREM}
The linear independence condition of the lemma is equivalent to the statement that if $c_1,c_2,\ldots$ are any 
integers all but finitely many vanishing and satisfying $\sum_{n=1}^\infty c_n=0$ and 
$\sum_{n=1}^\infty c_n\lambda_n=0$, then $c_1=c_2=\ldots=0$.
This is also equivalent to the statement that the pairwise differences among $\lambda_1,\lambda_2,\ldots$ do not
satisfy any non-trivial linear relation over $\Q$, i.e.\ if $c_{jk}$ for $1\leq j<k$ are integers all but finitely
many vanishing and satisfying $\sum_{j<k}c_{jk}(\lambda_j-\lambda_k)=0$, then
$\sum_{k=j+1}^\infty c_{jk}=\sum_{k=1}^{j-1}c_{k j}$ for all $j\geq1$.
\end{remark}
\begin{proof}
The fact that the equation $f(\sigma)=2e^{-\lambda_1\sigma}$ has exactly one real root $\sigma=\sigma_f>\sigma_0$
follows since the function $\sigma\mapsto e^{\lambda_1\sigma}f(\sigma)$ for $\sigma>\sigma_0$
is strictly decreasing, tends to $1$ as $\sigma\to\infty$, 
and by assumption tends to a limit which is greater than $2$ as $\sigma\to\sigma_0^+$.
For any $s$ with $\Re s>\sigma_f$ we have
$|\sum_{j=2}^\infty e^{-\lambda_js}|\leq\sum_{j=2}^\infty e^{-\lambda_j\Re s}<e^{-\lambda_1\Re s}
=|e^{-\lambda_1s}|$, and thus $f(s)\neq0$.
Hence it now only remains to prove that $f(s)$ has infinitely many zeros in any strip $\sigma_1<\Re s<\sigma_2$
with $\sigma_0\leq\sigma_1<\sigma_2\leq\sigma_f$.
Assume the contrary; then there even exist some $\sigma_1,\sigma_2$ with $\sigma_0\leq\sigma_1<\sigma_2\leq\sigma_f$
such that $f(s)$ has \textit{no} zero in the strip $\sigma_1<\Re s<\sigma_2$.
By basic facts in complex analysis, this implies that 
\begin{align}\label{BASICLEM1PF1}
\inf_{t\in\R}|f(\sigma+it)|>0
\end{align}
for any fixed $\sigma\in(\sigma_1,\sigma_2)$ 
(cf.\ \cite[\S 4 (Hilfssatz 3)]{BohrJessen} and \cite[\S 3 (Hilfssatz 3)]{Jessen1933}).

On the other hand, for any $\sigma\in(\sigma_1,\sigma_2)$ we have $\sum_{j=2}^\infty e^{(\lambda_1-\lambda_j)\sigma}>1$,
and hence there exist $\zeta_2,\zeta_3,\ldots\in\C$ satisfying $|\zeta_2|=|\zeta_3|=\ldots=1$ and
$\sum_{j=2}^\infty e^{(\lambda_1-\lambda_j)\sigma}\zeta_j=-1$.
It follows from Kronecker's density theorem (cf., e.g., \cite[Prop.\ 1.5.1]{KH}), using our linear independence assumption, that
for any given $J\in\Z_{\geq2}$ and $\ve>0$ there exists $t\in\R$ such that
$|e^{(\lambda_1-\lambda_j)it}-\zeta_j|<\ve$ for all $j\in\{2,\ldots,J\}$.
Applying this with $J\to\infty$ and $\ve\to0$, we conclude that
\begin{align*}
\inf_{t\in\R}|f(\sigma+it)|=
e^{-\lambda_1\sigma}\inf_{t\in\R}\Bigl|1+\sum_{j=2}^\infty e^{(\lambda_1-\lambda_j)(\sigma+it)}\Bigr|=0.
\end{align*}
This contradicts \eqref{BASICLEM1PF1}, and hence the lemma is proved.
\end{proof}

Recall the definition of the set $\widehat L$ from just below equation \eqref{primitiverelation2}.
We now prove:

\begin{lem}\label{BASICLEM2}
For each $n\geq2$ and %
almost all $L\in X_n$ the following holds:
The vector lengths $|\vecv|$ for $\vecv\in\widehat L$ are all distinct, 
and the pairwise differences of their logarithms do
not satisfy any non-trivial linear relation over $\Q$.
\end{lem}

\begin{proof}

We realize $X_n$ as the homogeneous space $\SL(n,\Z)\backslash \SL(n,\R)$, where $\SL(n,\Z)g$ corresponds to the lattice $\Z^ng\subset\R^n$. Note that $\mu_n$ is the unique probability measure on $X_n$ induced from a Haar measure on $\SL(n,\R)$. We will let $\mu_n$ denote also the corresponding Haar measure on $\SL(n,\R)$.
Now the statement of the lemma is equivalent to the following:
For almost every matrix $M\in\SL(n,\R)$, any finite %
sequence of primitive vectors $\vecu_1,\ldots,\vecu_N\in\Z^n$ ($N\geq2$) with $\vecu_j\neq\pm\vecu_k$ for $j\neq k$,
and any $b_1,\ldots,b_N\in\Z\setminus\{0\}$ with $\sum_{j=1}^Nb_j=0$, we have 
\begin{align}\label{nonvanish}
\sum_{j=1}^{N}b_j\log\big|\vecu_jM\big|\neq0.
\end{align}

Since there are only countably many possible $2N$-tuples $(\vecu_1,\ldots,\vecu_N,b_1,\ldots,b_N)$ it suffices to prove that for each \textit{fixed} choice of $(\vecu_1,\ldots,\vecu_N,b_1,\ldots,b_N)$ the set of $M\in\SL(n,\R)$ satisfying \eqref{nonvanish} has full measure in $\SL(n,\R)$.
We note that \eqref{nonvanish} is equivalent to $\prod_{j=1}^N\big|\vecu_jM\big|^{2b_j}\neq1$, i.e.
\begin{align}\label{nonvanish2}
\prod_{\substack{j=1\\(b_j>0)}}^N\big|\vecu_jM\big|^{2b_j}-\prod_{\substack{j=1\\(b_j<0)}}^N\big|\vecu_jM\big|^{2|b_j|}\neq0.
\end{align}
Hence, by the explicit formula for the measure $\mu_n$ on $\SL(n,\R)$ in terms of the matrix entries (cf., e.g., \cite{terras}) 
and the fact 
that  the left-hand side of \eqref{nonvanish2} is homogeneous in $M$ (since $\sum_{j=1}^Nb_j=0$), we find that it is enough to prove that for any fixed choice of $(\vecu_1,\ldots,\vecu_N,b_1,\ldots,b_N)$ as above, the relation \eqref{nonvanish2}
holds for Lebesgue almost all matrices $M\in\mathrm{Mat}_{n,n}(\R)\cong\R^{n^2}$. 

Note that each factor $|\vecu_jM|^2$ in \eqref{nonvanish2} is a real polynomial in the $n^2$ matrix entries of $M$.
Our conditions on $\vecu_1,\ldots,\vecu_N$ imply in particular that $\vecu_1$ is not proportional to any of
$\vecu_2,\ldots,\vecu_N$, and thus the set $S$ of vectors in $\R^n$ which are orthogonal to $\vecu_1$ 
but not orthogonal to any of $\vecu_2,\ldots,\vecu_N$ is non-empty.
Now, if we take any $M\in\mathrm{Mat}_{n,n}(\R)$ all of whose column vectors lie in $S$, we note that the left-hand side of
\eqref{nonvanish2} is non-zero.
This proves that the left-hand side of \eqref{nonvanish2} is a real, \textit{non-zero}
polynomial in the $n^2$ matrix entries of $M$.
Hence the condition \eqref{nonvanish2} is indeed fulfilled for Lebesgue almost all $M\in\mathrm{Mat}_{n,n}(\R)$,
and the lemma is proved.
\end{proof}

\begin{proof}[Proof of Proposition \ref{BASICPROP}]
Take any $L\in X_n$ such that the vector lengths $|\vecv|$ for $\vecv\in\widehat L$ are all distinct, 
and the pairwise differences of their logarithms do
not satisfy any non-trivial linear relation over $\Q$.
By Lemma \ref{BASICLEM2} this holds for almost every $L$.
Having fixed any such $L$, we consider the Dirichlet series $f(s)=\sum_{\vecv\in\widehat L}|\vecv|^{-2s}$.
The abscissa of (absolute) convergence for $f(s)$ is $\sigma_0=\frac n2$,
and we have $\lim_{\sigma\to{\frac n2}^+}f(\sigma)=\infty$; cf.\ \eqref{primitiverelation2}.
Hence, by Lemma \ref{BASICLEM1}, the supremum of the real parts of the zeros of $f(s)$
equals the unique real root $\sigma_f>\frac n2$ of the equation $f(\sigma)=2|\vecv_1|^{-2\sigma}$, where $\vecv_1$ is the
shortest vector in $\widehat L$; in fact $f(s)$ has infinitely many zeros in any strip
$\sigma_1<\Re s<\sigma_2$ with $\frac n2\leq\sigma_1<\sigma_2\leq\sigma_f$.
Using $|\vecv_1|=m(L)$ and \eqref{primitiverelation2}, we see that the equation for 
$\sigma_f$ may equivalently be expressed as $2m(L)^{-2\sigma}=\frac12\zeta(2\sigma)^{-1}E_n(L,\sigma)$.
Finally, using \eqref{primitiverelation2} and the fact that $\zeta(s)$ does not have any zeros when $\Re s>1$,
we see that $E_n(L,s)$ has exactly the same zeros (also counting multiplicity) as $f(s)$ in the half-plane
$\Re s>\frac n2$.
This completes the proof of Proposition \ref{BASICPROP}.
\end{proof}

\begin{remark}\label{TILDESIGMALSMOOTHREMARK}
Consider the function
\begin{align*}
L\mapsto\tilde\sigma_L:=\Bigl[\sigma>\tfrac n2 \text{ such that }2m(L)^{-2\sigma}=\tfrac12\zeta(2\sigma)^{-1}E_n(L,\sigma)\Bigr].
\end{align*}
(Thus Proposition \ref{BASICPROP} says that $\sigma_L=\tilde\sigma_L$ almost everywhere.)
Let $X_n'$ be the (closed) subset of $X_n$ consisting of those lattices for which $\#\left\{\vecv\in L\col |\vecv|=m(L)\right\}>2$; %
by \cite[Lemma 5.1]{poisson2}, $X_n'$ has measure zero.
We claim that $\tilde\sigma_L$ is a smooth function from $X_n\setminus X_n'$ to $\R_{>n/2}$.
This follows by studying the following function on $\R_{>n/2}\times X_n$:
\begin{align*}
\alpha(\sigma,L)=2m(L)^{-2\sigma}-\tfrac12\zeta(2\sigma)^{-1}E_n(L,\sigma)
=2m(L)^{-2\sigma}-\sum_{\vecv\in\widehat L}|\vecv|^{-2\sigma}.
\end{align*}
This function is smooth on all $\R_{>n/2}\times(X_n\setminus X_n')$ and 
one easily checks that $\frac{\partial}{\partial\sigma}\alpha(\sigma,L)>0$ at all points
$\sigma=\tilde\sigma_L$, $L\in X_n\setminus X_n'$.
Hence our smoothness claim follows from the implicit function theorem.

On the other hand note that $\tilde\sigma_L\to\infty$ whenever $L\to L_0$ for some $L_0\in X_n'$.
In particular this shows, via Proposition \ref{BASICPROP}, that $\sup_{L\in X_n}\sigma_L=\infty$.
(But of course, as we remarked in the introduction, $\sigma_L$ is finite for any fixed $L\in X_n$,
in particular for any $L\in X_n'$!)
\end{remark}

\begin{remark}\label{RATIONALNONGENERICREMARK}
Note that for any \textit{rational} $L\in X_n$, there occur arbitrarily large multiplicities among the
lengths $|\vecv|$ for $\vecv\in\widehat L$; in particular, the independence condition in Lemma \ref{BASICLEM2} \textit{fails} for every rational $L\in X_n$.
Indeed, for $n\geq3$ this claim follows easily from the fact that the number of $\vecv\in\widehat L$ with
$|\vecv|\leq R$ grows like $R^n$ as $R\to\infty$,
while the number of possible values of $|\vecv|$ %
grows at most like $R^2$ %
(since $L$ rational implies that  %
there is some $c>0$ such that $|\vecv|^2\in c\Z$ for all $\vecv\in L$).
The claim also holds for $n=2$, since in this case the number of possible values of $|\vecv|$
with $|\vecv|\leq R$ is in fact $\ll R^2(\log R)^{-\frac12}$ as $R\to\infty$;
cf.\ \cite{bernays} or \cite{pall}.
\end{remark}

\section{Proof of Theorem \ref{mainresult}}
\label{MAINRESULT1PROOFSEC}

The sequence $\{T_j\}_{j=1}^{\infty}$ of points of the Poisson process $\mathcal P$ belongs to the space 
\begin{align*}
\Omega:=\Big\{\vecx=\{x_j\}_{j=1}^{\infty}\in(\R_{>0})^{\infty}\col 0<x_1<x_2<x_3<\ldots
\Big\}, 
\end{align*}
which we equip with the subspace topology induced from the product topology on $(\R_{>0})^{\infty}$. We denote the distribution of $\mathcal{P}$ on $\Omega$ by $\PP$; this is a Borel probability measure on $\Omega$.

Recall the definition \eqref{SIGMATJDEF} of $\sigma_{\{T_j\}}$;
let us prove that this is a well-defined random variable on $(\Omega,\PP)$.
We set
\begin{align*}
\Omega':=\Bigl\{\vecx=\{x_j\}_{j=1}^{\infty}\in\Omega\col \#\{x_j<X\}\sim\sfrac12X\:\text{ as }\: X\to\infty\Bigr\}.
\end{align*}
This is a Borel subset of $\Omega$ %
and, by the strong law of large numbers, we have $\PP\Omega'=1$.
For any $\vecx\in\Omega'$, we have $\sum_{j=1}^\infty x_j^{-2\sigma}<\infty$ for all $\sigma>\frac12$ and
$\sum_{j=1}^\infty x_j^{-2\sigma}\to\infty$ as $\sigma\to{\frac12}^+$,
and thus, by the same argument as in the proof of Lemma \ref{BASICLEM1},
there exists a unique $\sigma>\frac12$ satisfying $x_1^{-2\sigma}=\sum_{j=2}^\infty x_j^{-2\sigma}$.
In other words, $\sigma_\vecx$ is well-defined for every $\vecx\in\Omega'$.
Furthermore, for any $c>\frac12$, we have
$\{\vecx\in\Omega'\col\sigma_\vecx<c\}=\{\vecx\in\Omega'\col x_1^{-2c}-\sum_{j=2}^\infty x_j^{-2c}>0\}$,
which is a Borel set.
This proves that the function $\vecx\mapsto\sigma_\vecx$ is $\PP$-measurable on $\Omega$,
i.e.\ that $\sigma_{\{T_j\}}$ is indeed a well-defined random variable.

For given $n\geq2$ and $c>\frac12$, we let $F_n(L,c)$ be the random variable given by
\begin{align}\label{FnLcDEF}
F_{n}(L,c):=-\scrV_1(L)^{-2c}+\sum_{j=2}^{\infty}\scrV_{j}(L)^{-2c},
\end{align}
where as usual $L$ is taken at random in $X_n$ according to $\mu_n$.
We also let $F(c)$ be the random variable
\begin{align}\label{FcDEF}
F(c):=-T_1^{-2c}+\sum_{j=2}^{\infty}T_{j}^{-2c}.
\end{align}

\begin{lem}\label{convergence}
Let $c>\frac12$ be fixed. Then $F_n(L,c)$ converges in distribution to $F(c)$ as $n\to\infty$.
\end{lem}
\begin{proof}
The proof is a straightforward adaptation of the proof of \cite[Thm.\ 1]{epstein1} with $m=1$. 
\end{proof}

\begin{lem}\label{BOUNDARYOKLEM}
For any given $c>\frac12$ and $\tau\in\R$, we have $\PP\bigl\{F(c)=\tau\bigr\}=0$.
\end{lem}
\begin{proof}
The lemma follows immediately from the calculations in the first paragraph of the proof of Theorem \ref{EXPLFORMULATHM}
(cf.\ p.\ \pageref{EXPLFORMULATHMPF} below).
\end{proof}

\begin{proof}[Proof of Theorem \ref{mainresult}]
It suffices to prove that for any fixed $c>\frac12$,
$Prob_{\mu_n}(n^{-1}\sigma_L>c)$ tends to $\PP(\sigma_{\{T_j\}}>c)$ as $n\to\infty$.
By Proposition \ref{BASICPROP} %
and the monotonicity argument 
at the beginning of the proof of Lemma \ref{BASICLEM1},
we have
\begin{align*}
Prob_{\mu_n}(n^{-1}\sigma_L>c)
&=Prob_{\mu_n}\Bigl\{L\in X_n\col -2\scrV_1(L)^{-2c}+\zeta(2cn)^{-1}\sum_{j=1}^\infty\scrV_j(L)^{-2c}>0\Bigr\}
\\
&=Prob_{\mu_n}\Bigl\{L\in X_n\col F_n(L,c)>\bigl(1-\zeta(2cn)^{-1}\bigr)\sum_{j=1}^\infty\scrV_j(L)^{-2c}\Bigr\}.
\end{align*}

Now let $\ve>0$ be given. By Lemma \ref{BOUNDARYOKLEM} we have $\PP\bigl\{F(c)=0\bigr\}=0$,
and hence (using \cite[Thm.\ 1.19(e)]{rudin}) there exists $\tau>0$ such that 
\begin{align}\label{mainresultpf1}
\PP\bigl\{F(c)\in[0,\tau]\bigr\}<\ve.
\end{align}
Furthermore, it follows from \cite{epstein1} that there exists $K>0$ and $N\in\Z^+$ such that%
\begin{align*}
Prob_{\mu_n}\Big\{L\in X_n\col\sum_{j=1}^{\infty}\scrV_j(L)^{-2c}< K\Big\}>1-\ve,
\qquad\forall n\geq N.
\end{align*}
After possibly increasing $N$, we may also assume that $(1-\zeta(2cn)^{-1})K<\tau$ for all $n\geq N$.
It follows that, for all $n\geq N$,
\begin{align*}
Prob_{\mu_n}\bigl(F_n(L,c)>\tau\bigr)-\ve
\leq Prob_{\mu_n}\bigl(n^{-1}\sigma_L>c\bigr)\leq Prob_{\mu_n}\bigl(F_n(L,c)>0\bigr).
\end{align*}
However, by Lemma \ref{convergence} and Lemma \ref{BOUNDARYOKLEM}, we have
(cf.\ \cite[Thm.\ 2.1(v)]{billconv})
\begin{align*}
\lim_{n\to\infty}Prob_{\mu_n}\bigl(F_n(L,c)>0\bigr)=\PP\bigl(F(c)>0\bigr)=\PP\bigl(\sigma_{\{T_j\}}>c\bigr)
\end{align*}
and
\begin{align*}
\lim_{n\to\infty}Prob_{\mu_n}\bigl(F_n(L,c)>\tau\bigr)=\PP\bigl(F(c)>\tau\bigr).
\end{align*}
Furthermore, by \eqref{mainresultpf1} we have $\PP\bigl(F(c)>\tau\bigr)>\PP\bigl(F(c)>0\bigr)-\ve
=\PP\bigl(\sigma_{\{T_j\}}>c\bigr)-\ve$.
Hence we obtain
\begin{align*}
\limsup_{n\to\infty}Prob_{\mu_n}\bigl(n^{-1}\sigma_L>c\bigr)\leq\PP\bigl(\sigma_{\{T_j\}}>c\bigr)
\end{align*}
and
\begin{align*}
\liminf_{n\to\infty}Prob_{\mu_n}\bigl(n^{-1}\sigma_L>c\bigr)\geq\PP\bigl(\sigma_{\{T_j\}}>c\bigr)-2\ve.
\end{align*}
But $\ve$ is arbitrary and hence the proof of Theorem \ref{mainresult} is complete.
\end{proof}

\section{Proof of Theorem \ref{ZERODENSITYLIMITTHM}}
\label{ZERODENSITYLIMITTHMSEC}

\begin{lem}\label{JESSEN1933MAINRESlem}
Let $f(s)=\sum_{j=1}^\infty e^{-\lambda_j s}$ be a Dirichlet series with real exponents $\lambda_1<\lambda_2<\cdots$
and abscissa of absolute convergence $\sigma_0<\infty$,
and set, for $\sigma>\sigma_0$,
\begin{align}\label{JESSEN1933MAINRESexpl}
\nu(\{\lambda_j\};\sigma):=
\frac1{2\pi}\sum_{a=1}^\infty\sum_{b=1}^\infty
(-1)^{\one(a\neq b)}
\lambda_{a}\lambda_{b}e^{-(\lambda_{a}+\lambda_{b})\sigma}
\int_0^\infty
\biggl(\prod_{j=1}^\infty J_{[j;a,b]}(e^{-\lambda_j\sigma}r)\biggr)\,r\,dr,
\end{align}
where $\one(\cdot)$ is the indicator function;
$[j;a,b]:=1$ if $j\in\{a,b\}$ and $a\neq b$, otherwise $[j;a,b]:=0$;
and $J_\alpha(x)$ is the Bessel function of %
order $\alpha\in\{0,1\}$.
Let $N(\sigma_1,\sigma_2;\tau_1,\tau_2)$ be the number of zeros of $f(s)$ in the rectangle
$s\in(\sigma_1,\sigma_2)\times(\tau_1,\tau_2)$, counting multiplicity.
If $\lambda_1,\lambda_2,\ldots$ are linearly independent over $\Q$,
then, for any fixed $\sigma_1<\sigma_2$ in $\R_{>\sigma_0}$,
\begin{align}\label{JESSEN1933MAINRESres}
\lim_{\tau_2-\tau_1\to\infty}\frac{N(\sigma_1,\sigma_2;\tau_1,\tau_2)}{\tau_2-\tau_1}
=\int_{\sigma_1}^{\sigma_2}\nu(\{\lambda_j\};\sigma)\,d\sigma.
\end{align}
\end{lem}
\begin{proof}
This follows from Jessen \cite{Jessen1934};
cf.\ in particular %
\cite[Sec.\ 28]{Jessen1934}
and the explicit formula for 
$G(\sigma,z)$ in \cite[Sec.\ 24]{Jessen1934}
(applied with $z=0$; we evaluate the Fourier integral in
\cite[p.\ 310 (line -10)]{Jessen1934} using the explicit formula for $\Psi$ found in the same section). %
The expression in \eqref{JESSEN1933MAINRESexpl} is nicely convergent for any $\sigma>\sigma_0$
and defines a continuous function of $\sigma$
(cf.\ \cite[Sec.\ 24]{Jessen1934} or \cite{wintner};
more details on convergence also appear in the proof of Lemma \ref{ZDLTHMkeylem} below).
\end{proof}

\begin{lem}\label{JESSENFORMULAINVlem}
Let $\nu(\{\lambda_j\};\sigma)$ be as in Lemma \ref{JESSEN1933MAINRESlem}.
Then $\nu(\{\lambda_j+\alpha\};\sigma)=\nu(\{\lambda_j\};\sigma)$ for any constant $\alpha\in\R$.
\end{lem}
\begin{proof}
If $\lambda_1,\lambda_2,\ldots$ are linearly independent over $\Q$,
and the same holds for $\lambda_1+\alpha,\lambda_2+\alpha,\ldots$,
then the claim follows from \eqref{JESSEN1933MAINRESres}, since $f(s)$ and $e^{-\alpha s}f(s)$
have the same zeros.
The general case follows by %
continuity (cf.\ \cite[end of Sec.\ 24]{Jessen1934}). %
\end{proof}
\begin{remark}
Lemma \ref{JESSENFORMULAINVlem} can also be proved directly from 
\eqref{JESSEN1933MAINRESexpl} by using the identity
\begin{align}\label{GZEXPLformaleq1PF2}
c_a\int_0^\infty\Bigl(\prod_{j=1}^\infty J_0(c_jr)\Bigr)\,r\,dr
=\sum_{b\neq a} c_b\int_0^\infty
J_1(c_ar)J_1(c_b r)\Bigl(\prod_{j\notin\{a,b\}}J_0(c_jr)\Bigr)\,r\,dr,
\end{align}
which holds for any $c_1,c_2,\ldots>0$ with $\sum_j c_j<\infty$, and any $a\in\Z^+$.
One proves \eqref{GZEXPLformaleq1PF2} %
by integration by parts,
using $c_arJ_0(c_ar)=\frac d{dr}(rJ_1(c_ar))$ and
$\frac d{dr} J_0(c_br)=-c_bJ_1(c_br)$.
\end{remark}

Now let $L\in X_n$ be any lattice which is generic in the sense of  %
Lemma \ref{BASICLEM2};
order the vectors of $\widehat L$ by increasing lengths as
$\wv_1,\wv_2,\ldots$, and set $\lambda_j=2\log|\wv_j|$,
so that $f(s)=\sum_{\vecv\in \widehat L}|\vecv|^{-2s}$ in Lemma \ref{JESSEN1933MAINRESlem}.
The condition of Lemma \ref{BASICLEM2} implies that for Lebesgue almost every $\alpha\in\R$,
the numbers $\lambda_1+\alpha,\lambda_2+\alpha,\ldots$ are linearly independent over $\Q$.
Hence, by \eqref{primitiverelation2} and Lemmas \ref{JESSEN1933MAINRESlem} and \ref{JESSENFORMULAINVlem},
relations \eqref{HLDEF} and \eqref{HLREL} hold for all $\sigma_1<\sigma_2$ in $(\frac n2,\infty)$,
with $\nu_L(\sigma)=\nu(\{\lambda_j\};\sigma)$.

\begin{remark}\label{EXACTSUPPrem}
For such a lattice $L\in X_n$,
it follows from  \eqref{HLDEF} and \eqref{HLREL} that $\nu_L(\sigma)=0$ for all $\sigma\geq\sigma_L$.
Furthermore, by Lemma \ref{BASICLEM1} and Jessen \cite[Satz B]{Jessen1933}, 
$\nu_L(\sigma)$ does not vanish identically on any subinterval of $(\frac n2,\sigma_L)$.
Hence the limit in \eqref{HLDEF} is positive whenever $\sigma_1<\sigma_L$.
This also implies that the support of $\nu_L$ in $(\frac n2,\infty)$ is exactly equal to $(\frac n2,\sigma_L]$.
\end{remark}

Next we set
\begin{align*}
\eta_j=\eta_j(L):=2\log V_n+n\lambda_j=2(\log V_n+n\log|\wv_j|).
\end{align*}
Then, again by Lemma \ref{JESSENFORMULAINVlem}, we have
\begin{align}\label{ZDLTHMpf2}
\nu(\{\eta_j\};\sigma)=n^2\nu_L(n\sigma).
\end{align}

\begin{lem}\label{POISSONCONVlem}
Fix any $K\in\Z^+$,
and take $L$ at random in $X_n$ according to $\mu_n$.
Then the random variable
$(e^{\eta_1/2},\ldots,e^{\eta_K/2})$ converges in distribution to $(T_1,\ldots,T_K)$ as $n\to\infty$.
\end{lem}
\begin{proof}
This is a simple variant of \cite[Thm.\ 1]{poisson}.
Indeed, with notation as in \eqref{volumes}, \cite[Thm.\ 1]{poisson} implies that
$\scrV_K(L)<2^n\scrV_1(L)$ holds with probability tending to $1$ as $n\to\infty$.
Hence, for any such $L$ 
the lattice vectors $\pm\vecv_1,\ldots,\pm\vecv_K$ are all primitive,
so that %
$e^{\eta_j/2}=\scrV_j(L)$ for $j=1,\ldots,K$.
\end{proof}

For any $K\geq5$, we set
\begin{align}\notag%
\nu^{(K)}(\{\lambda_j\};\sigma):=
\frac1{2\pi}\sum_{a=1}^K\sum_{b=1}^K
(-1)^{\one(a\neq b)}
\lambda_{a}\lambda_{b}e^{-(\lambda_{a}+\lambda_{b})\sigma}
\int_0^\infty
\biggl(\prod_{j=1}^K J_{[j;a,b]}(e^{-\lambda_j\sigma}r)\biggr)\,r\,dr.
\end{align}
Given any interval $I\subset\R$, we let $C(I)$ be the space of real-valued continuous functions on $I$,
provided with the supremum norm $\|\cdot\|_{C(I)}$.
\begin{lem}\label{ZDLTHMkeylem}
Let $\eta_j=\eta_j(L)=2(\log V_n+n\log|\wv_j|)$ as above.
Let $I$ be a compact subinterval of $(\frac12,\infty)$, and let $\ve>0$.
Then there exist integers $n_0\geq2$ and $K_0\geq5$ such that,
for all $K\geq K_0$ and $n\geq n_0$,
\begin{align*}
\mu_n\bigl(\bigl\{L\in X_n\col \bigl\|\nu(\{\eta_j\};\,\cdot\,)-\nu^{(K)}(\{\eta_j\};\,\cdot\,)\bigr\|_{C(I)}\leq\ve\bigr\}\bigr)
\geq1-\ve.
\end{align*}
\end{lem}
\begin{proof}
Fix some $c$ with $\frac12<c<\inf I$.
By Lemma \ref{POISSONCONVlem} and \cite[Thm.\ 1]{epstein1}, if we take $n_0$ and $A$ sufficiently large, then
for all $n\geq n_0$ we have 
\begin{align}\label{ZDLTHMkeylemPF2}
\mu_n\Bigl(\Bigl\{L\in X_n\col -A\leq\eta_1(L)<\eta_5(L)\leq A\:\text{ and }\:
\sum_{j=1}^\infty e^{-c\eta_j}<A\Bigr\}\Bigr)>1-\tfrac13\ve.
\end{align}
Note that $|\eta|\ll e^{(\sigma-c)\eta}$ uniformly over all $\eta\geq -A$ and $\sigma\in I$.
We set $\oJ(x):=\max(|J_0(x)|,|J_1(x)|)$; then $\oJ(x)\ll x^{-1/2}$ as $x\to\infty$.
Using these facts,
we conclude that there is some $B>0$ such that,
for any $n\geq n_0$ and any $L$ in the set in \eqref{ZDLTHMkeylemPF2},
\begin{align}\label{ZDLTHMkeylemPF1}
\sum_{j=1}^\infty |\eta_j| e^{-\eta_j\sigma}<B
\quad\text{and}\quad
\int_0^\infty \biggl|\prod_{j=1}^5 \oJ(e^{-\eta_j\sigma}r)\biggr|\,r\,dr<B,
\qquad\forall \sigma\in I.
\end{align}
For any $L$ satisfying \eqref{ZDLTHMkeylemPF1}, and $\sigma\in I$, we have,
since $\oJ(x)\leq1$ for all $x\geq0$,  %
\begin{align*}
\bigl|\nu(\{\eta_j\};\sigma)-\nu^{(K)}(\{\eta_j\};\sigma)\bigr|
\leq\frac{B^2}{\pi}\sum_{j>K}|\eta_j|e^{-\eta_j\sigma}
\hspace{160pt}
\\
+\frac{B^2}{2\pi}\int_0^\infty\biggl|1-\prod_{j>K}J_0(e^{-\eta_j\sigma}r)\biggr|\,
\biggl(\prod_{j=1}^5\oJ(e^{-\eta_j\sigma}r)\biggr)\,r\,dr.
\end{align*}
Hence it now suffices to prove that for any given $\ve'>0$ and $R>0$,
if we take $K$ and $n_0$ sufficiently large, then for all $n\geq n_0$ we have both
\begin{align}\label{ZDLTHMkeylemPF3}
\mu_n\biggl(\biggl\{L\in X_n\col\sup_{\sigma\in I}\sum_{j>K}|\eta_j|e^{-\eta_j\sigma}<\ve'\biggr\}\biggr)>1-\tfrac13\ve
\end{align}
and 
\begin{align}\label{ZDLTHMkeylemPF4}
\mu_n\biggl(\biggl\{L\in X_n\col\sup_{\sigma\in I}\sup_{r\in[0,R]}
\biggl|1-\prod_{j>K}J_0(e^{-\eta_j\sigma}r)\biggr|<\ve'\biggr\}\biggr)>1-\tfrac13\ve. %
\end{align}
Here \eqref{ZDLTHMkeylemPF3} is a consequence of (e.g.) \cite[Thm.\ 5]{epstein1}
(applied with $k=2$, $c$ as above, and $\delta$ sufficiently large).
To prove \eqref{ZDLTHMkeylemPF4},
note that $1\geq J_0(x)=1+O(x^2)$ as $x\to0$; hence there is a constant $\alpha>0$ such that
$e^{-x}\leq J_0(x)\leq1$ for all $x\in[0,\alpha]$.
It follows that $L$ belongs to the set in the left-hand side of \eqref{ZDLTHMkeylemPF4} whenever 
\begin{align*}
\sup_{\sigma\in I}\sum_{j>K}e^{-\eta_j\sigma}<R^{-1}\min(\alpha,|\log(1-\ve')|).
\end{align*}
Using this observation, \eqref{ZDLTHMkeylemPF4} follows by another application of \cite[Thm.\ 5]{epstein1}. This completes the proof of the lemma.
\end{proof}

\begin{proof}[Proof of Theorem \ref{ZERODENSITYLIMITTHM}]
Consider the random function $\nu_{\{T_j\}}$ (cf.\ \eqref{NUTJdef}).
By Lemma \ref{JESSEN1933MAINRESlem},
$\nu_{\{T_j\}}(\sigma)=\nu(\{2\log T_j\};\sigma)$ (almost surely),
and %
one easily verifies that 
\begin{align}\label{ZDLTHMpf1}
\nu_{\{T_j\}}=\lim_{K\to\infty}\nu^{(K)}(\{2\log T_j\};\,\cdot\,)
\quad\text{in }\: C(\tfrac12,\infty) \:\text{ (almost surely);}
\end{align}
cf.\ \cite[Sec.\ 24]{Jessen1934} or \cite{wintner},
or the proof of Lemma \ref{ZDLTHMkeylem}.
This shows in particular that $\nu_{\{T_j\}}$
is a measurable map from $(\Omega,\PP)$ to $C(\frac12,\infty)$, viz., a random element in $C(\frac12,\infty)$.

Now, for any fixed $K\geq5$, $\:\nu^{(K)}(\{\lambda_j\},\,\cdot\,)\:$ is a continuous function 
of $(\lambda_1,\ldots,\lambda_K)\in(\R_{>0})^K$ with values in
$C(\frac12,\infty)$;
and by Lemma \ref{POISSONCONVlem},
$(\eta_1,\ldots,\eta_K)$ converges in distribution to $(2\log T_1,\ldots,2\log T_K)$ as $n\to\infty$.
Therefore $\nu^{(K)}(\{\eta_j\},\,\cdot\,)$ converges in distribution to
$\nu^{(K)}(\{2\log T_j\};\,\cdot\,)$.
Using this fact, %
\eqref{ZDLTHMpf1} and Lemma \ref{ZDLTHMkeylem},
it follows that for any fixed compact interval $I\subset(\frac12,\infty)$,
the restriction of $n^2\nu_L(n\,\cdot\,)=\nu(\{\eta_j\},\,\cdot\,)$ to $I$ converges in distribution to 
the restriction of $\nu_{\{T_j\}}$ to $I$, %
as random elements in $C(I)$;  %
cf.\ \cite[Thm.\ 4.28]{Kallenberg}.
Hence, by \cite[Prop.\ 16.6]{Kallenberg}, convergence also holds %
in $C(\frac12,\infty)$,
i.e.\ we have proved that $n^2\nu_L(n\,\cdot\,)\xrightarrow[]{\textup{ d }}\nu_{\{T_j\}}$ as $n\to\infty$,
in the sense of convergence in distribution for random elements in %
$C(\frac12,\infty)$.

To complete the proof of Theorem \ref{ZERODENSITYLIMITTHM},
it remains to upgrade the result to 
\textit{joint} convergence of $n^2\nu_L(n\,\cdot\,)$ and $n^{-1}\sigma_L$.
Given the previous arguments, this is a standard but %
somewhat technical exercise:
Recalling \eqref{FnLcDEF} and \eqref{FcDEF}, we set
\begin{align*}
F_n^{(K)}(L,c):=-\scrV_1(L)^{-2c}+\sum_{j=2}^K\scrV_{j}(L)^{-2c}
\quad\text{and}\quad
F^{(K)}(c):=-T_1^{-2c}+\sum_{j=2}^{K}T_{j}^{-2c}.
\end{align*}
The key fact, now, is that for any fixed $c>\frac12$ and $K\geq5$,
the following convergence in distribution of random elements in $C(\frac12,\infty)$ holds,
as $n\to\infty$:
\begin{align*}
\one\bigl(F_n^{(K)}(L,c)>0\bigr)\,\,\nu^{(K)}(\{\eta_j\},\,\cdot\,)
\xrightarrow[]{\textup{ d }}
\one\bigl(F^{(K)}(c)>0\bigr)\,\,\nu^{(K)}(\{2\log T_j\};\,\cdot\,).
\end{align*}
Indeed, by the proof of Lemma~\ref{POISSONCONVlem},
away from a set of $L\in X_n$ of measure tending to zero as $n\to\infty$,
$F_n^{(K)}(L,c)$ is a continuous function of $(\eta_1,\ldots,\eta_K)$,
just as $\nu^{(K)}(\{\eta_j\},\,\cdot\,)$;
also $\PP(F^{(K)}(c)=0)=0$; %
hence the claim follows from Lemma \ref{POISSONCONVlem} and the mapping theorem
\cite[Thm.\ 2.7]{billconv}.
Furthermore,
from the proofs of Theorem \ref{mainresult} and \cite[Thm.\ 1]{epstein1},
one extracts the fact that for given $c$ and $\ve>0$, 
if $K$ and $n$ are taken sufficiently large 
and $L$ is picked at random in $(X_n,\mu_n)$, then
with probability greater than $1-\ve$,
the two inequalities $n^{-1}\sigma_L>c$ and $F_n^{(K)}(L,c)>0$ are both true or both false.
Using these facts and previous arguments (in particular Lemma \ref{ZDLTHMkeylem}),
we may again apply \cite[Thm.\ 4.28]{Kallenberg} to conclude that,
for any fixed compact interval $I\subset(\frac12,\infty)$ and $c>\frac12$,
\begin{align*}
\one\bigl(n^{-1}\sigma_L>c\bigr)\,n^2\nu_L(n\,\cdot\,)_{|I}
\xrightarrow[]{\textup{ d }}
\one\bigl(\sigma_{\{T_j\}}>c\bigr)\,\nu_{\{T_j\}\,|I}
\qquad
\text{as }\: n\to\infty
\end{align*}
(convergence in distribution of random elements in $C(I)$).
Finally, Theorem \ref{ZERODENSITYLIMITTHM} follows
by general measure-theoretic arguments
(akin to \cite[Thms.\ 2.3, 2.4]{billconv}).
\end{proof}

\begin{remark}
Although $\sigma_L=\sup(\supp(\nu_L))$ for generic $L\in X_n$
(cf.\ Remark \ref{EXACTSUPPrem}),
and $\sigma_{\{T_j\}}=\sup(\supp(\nu_{\{T_j\}}))$ almost surely,
it does not seem that the joint convergence of Theorem \ref{ZERODENSITYLIMITTHM} 
follows in any automatic way from 
just knowing $n^2\nu_L(n\,\cdot\,)\xrightarrow[]{\textup{ d }}\nu_{\{T_j\}}$.
Note in particular that the 
map $\nu\mapsto\sup(\supp(\nu))$ from $C(\frac12,\infty)$ to $\R_{>1/2}\cup\{\pm\infty\}$ 
is far from continuous.
\end{remark}

\section{Proof of Theorem \ref{EXPLFORMULATHM}}
\label{EXPLFORMULATHMSEC}

In this section we prove Theorem \ref{EXPLFORMULATHM}, which
gives an explicit formula for %
the distribution function of $\sigma_{\{T_j\}}$.
Recall that $\sigma_{\{T_j\}}>c$ holds if and only if $\sum_{j=2}^\infty T_j^{-2c}>T_1^{-2c}$.
Hence our task is to determine the probability
\begin{align}\label{EXPLFORMULATHMPF1}
\PP\bigl(\sigma_{\{T_j\}}>c\bigr)=
\PP\Bigl(\sum_{j=2}^\infty T_j^{-2c}>T_1^{-2c}\Bigr).
\end{align}
Let us note that for any fixed $\mu>0$, the sequence
$\mu T_1,\mu T_2,\ldots$ give the points of a Poisson process on the positive real line with intensity
$(2\mu)^{-1}$, and we have
$\sum_{j=2}^\infty T_j^{-2c}>T_1^{-2c}$ if and only if $\sum_{j=2}^\infty(\mu T_j)^{-2c}>(\mu T_1)^{-2c}$.
Hence we may, in order to make our computations slightly cleaner,
\textit{alter our notation} so that from now on, $0<T_1<T_2<\ldots$ denote the points of a 
Poisson process on the positive real line with constant intensity \textit{one;}
the probability in \eqref{EXPLFORMULATHMPF1} remains unchanged by this alteration.
Also to make the computations slightly cleaner, we will write 
\begin{align*}
a:=2c\in\R_{>1}.
\end{align*}

As a first step, we consider the \textit{conditional} distribution of the sum $\sum_{j=2}^\infty T_j^{-a}$ 
given the value of $T_1$.
We will see that this distribution is infinitely divisible.
For basic facts about infinitely divisible distributions, cf., e.g., \cite[Chs.\ VI.3, IX, XVII]{feller}.
We formulate the result for a Poisson process having constant intensity $1$; it is of course easy to 
carry this over to the case of an arbitrary constant intensity.

\begin{prop}\label{INFDIVEXPLPROP}
Let $0<T_1<T_2<\cdots$ be the points of a Poisson process on the positive real line with constant intensity $1$.
Then, for any $a>1$ and $\delta>0$, the conditional distribution of $\sum_{j=2}^\infty T_j^{-a}$, given that $T_1=\delta$,
is an infinitely divisible distribution, the characteristic function of which is given by
\begin{align}\label{INFDIVEXPLPROPRES}
\varphi_{a,\delta}(t)
=\EE\Bigl( e^{it\sum_{j=2}^\infty T_j^{-a}} \:\big|\: T_1=\delta\Bigr)
=\exp\biggl\{-\int_\delta^{\infty} \bigl(1-e^{itx^{-a}}\bigr)\,dx\biggr\}.
\end{align}
\end{prop}
(Cf.\ \cite[Thm.\ 1.4.2]{SamorodnitskyTaqqu}, where the corresponding fact is proved in the special case $\delta=0$
but with more general weights in the sum; the resulting distribution is then a stable distribution.)
\begin{proof}
Let $n$ be a positive integer and let $\eta$ be any real number larger than $\delta$.
The conditional distribution of $(T_2,\ldots,T_{n+1})$, given that 
$T_1=\delta$ and $T_{n+2}=\eta$, 
is that of the order statistic of $n$
i.i.d.\ random variables uniformly distributed in the interval $(\delta,\eta)$,
and hence the conditional distribution of $\sum_{j=2}^{n+1}T_j^{-a}$, given $T_1=\delta$ and $T_{n+2}=\eta$,
is the same as the distribution of
$\sum_{j=1}^n (\delta+(\eta-\delta)U_j)^{-a}$, where from now on $U_1,U_2,\ldots$ denotes a sequence of
i.i.d.\ random variables uniformly distributed in $(0,1)$.
It follows that the conditional distribution of $\sum_{j=2}^{n+1}T_j^{-a}$, given only $T_1=\delta$,
is the same as the distribution of 
\begin{align*}
X_n:=\sum_{j=1}^n (\delta+S_{n+1}U_j)^{-a},
\end{align*}
where $S_{n+1}$ denotes the sum of $n+1$ i.i.d.\ exponential random variables with mean one, independent from 
the sequence $\{U_j\}$
(so that $S_{n+1}$ has the same distribution as $T_{n+2}-\delta$ given $T_1=\delta$).

By the law of large numbers $n^{-1}S_{n+1}$ tends in distribution to $1$,
i.e.\ given any $\ve>0$ there is $N\in\Z_{>0}$ such that for each $n\geq N$,
we have $(1-\ve)n<S_{n+1}<(1+\ve)n$ with probability $>1-\ve$.
It follows that if we let
\begin{align*}
Y_n:=\sum_{j=1}^n (\delta+nU_j)^{-a},
\end{align*}
then, for each $n\geq N$, we have $(1+\ve)^{-a}Y_n <X_n<(1-\ve)^{-a}Y_n$ with probability $>1-\ve$.
In particular, it now suffices to prove that $Y_n$ tends in distribution to a 
(non-defective) random variable whose characteristic function is given by the right-hand side of \eqref{INFDIVEXPLPROPRES},
since then also $X_n$ must converge in distribution to this random variable,
and also it follows from the definition of $Y_n$ that the limit distribution must be infinitely divisible,
cf., e.g., \cite[Ch IX.5 (see also Ch.\ XVII.2)]{feller}.   %

But $Y_n$ is a sum of $n$ independent random variables, and thus its characteristic function equals
\begin{align*}
\EE e^{itY_n}=\Bigl(\EE e^{it(\delta+nU_1)^{-a}}\Bigr)^n
&=\biggl(\frac1n\int_\delta^{\delta+n} e^{itx^{-a}}\,dx\biggr)^n
=\biggl(1-\frac1n\int_\delta^{\delta+n} \bigl(1-e^{itx^{-a}}\bigr)\,dx\biggr)^n.
\end{align*}
Note that $|1-e^{itx^{-a}}|\ll |t|x^{-a}$ uniformly for all $x\geq\delta$ and all $t\in\R$.
In particular, for each fixed $t\in\R$ 
the integral $\int_\delta^{\infty} \bigl(1-e^{itx^{-a}}\bigr)\,dx$ is absolutely convergent, 
and $\EE e^{itY_n}$ tends to the expression in the right-hand side of \eqref{INFDIVEXPLPROPRES} as $n\to\infty$.
The bound  $|1-e^{itx^{-a}}|\ll |t|x^{-a}$ also implies that the function $\varphi_{a,\delta}(t)$ is continuous.
Hence $Y_n$ converges in distribution to a (non-defective) random variable whose characteristic function
is given by the right-hand side of \eqref{INFDIVEXPLPROPRES}, and the proposition is proved.
\end{proof}

\begin{remark}\label{INCOMPLETEGAMMAREMARK}
Let us note that the integral in \eqref{INFDIVEXPLPROPRES} 
may be expressed in terms of the incomplete gamma function.
Indeed, substituting $x=(iu)^{-\frac1a}$ %
and then integrating by parts, we get
\begin{align}\label{INCOMPLETEGAMMAFORMULA}
\int_\delta^{\infty} \bigl(1-e^{itx^{-a}}\bigr)\,dx
&=-\int_0^{-i\delta^{-a}}\bigl(1-e^{-tu}\bigr)\Bigl(\frac d{du}\bigl((iu)^{-\frac1a}\bigr)\Bigr)\,du
\nonumber\\
&=\delta(e^{it\delta^{-a}}-1)+t\int_0^{-i\delta^{-a}}e^{-tu}(iu)^{-\frac1a}\,du
\nonumber\\
&=\delta(e^{it\delta^{-a}}-1)+(-it)^{\frac1a}\gamma\Bigl(1-\tfrac1a,-it\delta^{-a}\Bigr),
\end{align}
where for $t\neq0$ we agree that $\arg(-it)=-(\sgn t)\frac\pi2$. 
Hence
\begin{align*}
\varphi_{a,\delta}(t)
=\exp\Bigl\{-\delta(e^{it\delta^{-a}}-1)-(-it)^{\frac1a}\gamma(1-\tfrac1a,-it\delta^{-a})\Bigr\}.
\end{align*}
Furthermore, using the recursion formula \eqref{INCOMPLETEGAMMAREC} together with the formula 
$\gamma(s,z)=\Gamma(s)-\Gamma(s,z)$,
where 
\begin{align}\label{INCGAMMADEF}
\Gamma(s,z):=\int_z^{\infty} u^{s-1}e^{-u}\,du
\end{align}
is the upper incomplete gamma function,
we get the alternative  formula
\begin{align}\label{INCOMPLETEGAMMAREPR}
\varphi_{a,\delta}(t)
=\exp\Bigl\{\delta-(-it)^{\frac1a}\Gamma(1-\tfrac1a)-\tfrac1a(-it)^{\frac1a}\Gamma(-\tfrac1a,-it\delta^{-a})\Bigr\}.
\end{align}
\end{remark}

\begin{proof}[Proof of Theorem \ref{EXPLFORMULATHM}]
Note that, for all $z\in \C\setminus\{0\}$ with $\Re z\geq0$,\label{EXPLFORMULATHMPF}
\begin{align*}
\left|\Gamma(-\tfrac1a,z)\right|\leq|z|^{-\frac1a-1}e^{-\Re z}.
\end{align*}
Hence, if we denote the exponent in \eqref{INCOMPLETEGAMMAREPR} by $\psi_{a,\delta}(t)$, we have for $t>0$,
\begin{align*}
\psi_{a,\delta}(t)=\delta-(-it)^{\frac1a}\Gamma(1-\tfrac1a)+O_{a,\delta}\left(t^{-1}\right).
\end{align*}
Using also $\Re (-it)^{\frac1a}\gg_at^{\frac1a}$, we conclude that $-\Re\psi_{a,\delta}(t)\gg_{a,\delta} t^{\frac1a}$
as $t\to\infty$.
Hence, in view of the symmetry $\varphi_{a,\delta}(-t)=\overline{\varphi_{a,\delta}(t)}$,
the function $\varphi_{a,\delta}$ is integrable, and therefore the 
distribution in Proposition \ref{INFDIVEXPLPROP} has a density function, which we call $f_{a,\delta}(x)$.
Thus
\begin{align*}
f_{a,\delta}(x)=\frac1{2\pi}\int_{-\infty}^\infty \varphi_{a,\delta}(t)e^{-itx}\,dt
=\frac1{\pi}\int_0^\infty\Re\bigl(\varphi_{a,\delta}(t)e^{-itx}\bigr)\,dt.
\end{align*}
It follows that the conditional probability of $\sum_{j=2}^\infty T_j^{-a}>T_1^{-a}$, given that $T_1=\delta$, is
\begin{align*}
\PP\Bigl(\sum_{j=2}^\infty T_j^{-a}>T_1^{-a}\:\:\big|\:\: T_1=\delta\Bigr)=\int_{\delta^{-a}}^\infty f_{a,\delta}(x)\,dx.
\end{align*}
However, $T_1$, being the first point of a Poisson process on the positive real line with intensity one,
has an exponential distribution of mean one.
Hence we conclude:
\begin{align}\notag
\PP\bigl(\sigma_{\{T_j\}}>c\bigr)=
\PP\Bigl(\sum_{j=2}^\infty T_j^{-a}>T_1^{-a}\Bigr)
&=\int_0^\infty\int_{\delta^{-a}}^\infty f_{a,\delta}(x)\,dx\,e^{-\delta}\,d\delta
\\\notag
&=\int_0^\infty\int_{x^{-\frac1a}}^\infty f_{a,\delta}(x)\,e^{-\delta}\,d\delta\,dx
\\\label{EXPLFORMULATHMPF2}
&=\frac1\pi\int_0^\infty\int_{x^{-\frac1a}}^\infty 
\int_0^\infty\Re\bigl(\varphi_{a,\delta}(t)e^{-itx-\delta}\bigr)\,dt\,d\delta\,dx.
\end{align}
Note that the last expression in \eqref{EXPLFORMULATHMPF2} should be viewed as an iterated integral;
it is easy to see that 
$\int_0^\infty\int_{x^{-\frac1a}}^\infty
\int_0^\infty\bigl|\Re\bigl(\varphi_{a,\delta}(t)e^{-itx-\delta}\bigr)\bigr|\,dt\,d\delta\,dx=\infty$,
so that we are not permitted to change order of integration arbitrarily.
However, we will prove that the inner double integral is absolutely convergent.

By Proposition \ref{INFDIVEXPLPROP} we have
$e^{-\delta}\varphi_{a,\delta}(t)=\exp\bigl(-\bigl(\delta+\int_\delta^{\infty}(1-e^{itx^{-a}})\,dx\bigr)\bigr)$, and here
we have, by substituting $x=(u/t)^{-\frac1a}$ and then integrating by parts,
\begin{align*}
\delta+\int_\delta^{\infty}(1-e^{itx^{-a}})\,dx
=\delta-t^{\frac1a}\int_0^{t\delta^{-a}}(1-e^{iu})\Bigl(\frac d{du}( u^{-\frac1a})\Bigr)\,du
\\
=\delta e^{it\delta^{-a}}-it^{\frac1a}\int_0^{t\delta^{-a}}e^{iu}u^{-\frac1a}\,du
=t^{\frac1a}\Phi_a(t\delta^{-a}),
\end{align*}
where we have defined
\begin{align}\label{PHIAYDEF}
\Phi_a(y):=y^{-\frac1a}e^{iy}-i\int_0^y e^{iu}u^{-\frac1a}\,du\qquad
\text{for }\:a>1,\: y>0.
\end{align}
Thus
\begin{align}\label{EXPLFORMULATHMPF3}
\PP\bigl(\sigma_{\{T_j\}}>c\bigr)
=\frac1\pi\int_0^\infty\int_{x^{-\frac1a}}^\infty\int_0^\infty\Re\exp\Bigl\{-itx-t^{\frac1a}\Phi_a(t\delta^{-a})\Bigr\}
\,dt\,d\delta\,dx.
\end{align}

Using $e^{iu}=1+O(u)$ for $u\in[0,1]$, we find that 
$\Phi_a(y)=y^{-\frac1a}(1+O(y))$ for $0<y\leq1$.
(Here, and in any ``big-$O$'' or "$\ll$" bound below, we allow the implied constant to depend on $a$.)
In particular there exists a positive number $\kappa_1$, which may depend on $a$, such that
$\Re\Phi_a(y)\geq\frac12 y^{-\frac1a}$ for all $y\in(0,\kappa_1]$.
We also note that
\begin{align}\label{PHIAYDER}
\Phi_a'(y)=-\frac1a y^{-1-\frac1a}e^{iy}.
\end{align}
In particular $\Re\Phi_a'(y)=-\frac1ay^{-1-\frac1a}\cos y$,
and this is negative for all $y\in(0,\frac\pi2)$, so that
$\Re\Phi_a(y)>\Re\Phi_a(\frac\pi2)$ holds for all $y\in(0,\frac\pi2)$.
Furthermore, for all $y\geq\frac\pi2$ we have
$\Re\Phi_a(y)=\Re\Phi_a(\tfrac\pi2)-\frac1a\int_{\frac\pi2}^y u^{-1-\frac1a}(\cos u)\,du\geq\Re\Phi_a(\tfrac\pi2)$.
Also note from \eqref{PHIAYDEF} that $\Re\Phi_a(\frac\pi2)=\int_0^{\frac\pi2}u^{-\frac1a}(\sin u)\,du>0$.
Hence we conclude:
\begin{align*}
\Re\Phi_a(y)\geq \kappa_2:=\Re\Phi_a(\tfrac\pi2)>0,\qquad\forall y>0.
\end{align*}
Using the bounds obtained, we conclude:
\begin{align}\label{EXPLFORMULATHMPF6}
\int_0^\infty\Bigl|\exp\Bigl\{-t^{\frac1a}\Phi_a(t\delta^{-a})\Bigr\}\Bigr|\,dt
\leq\int_0^{\kappa_1\delta^a}e^{-\frac12\delta}\,dt+\int_{\kappa_1\delta^a}^\infty e^{-\kappa_2t^{\frac1a}}\,dt
\ll e^{-\kappa_3\delta}
\end{align}
for all $\delta>0$, where $\kappa_3$ is some positive number which may depend on $a$.
From this estimate we see that the inner double integral in \eqref{EXPLFORMULATHMPF3} is indeed absolutely convergent,
in fact even $\int_0^\infty\int_0^\infty\bigl|\exp(-t^{\frac1a}\Phi_a(t\delta^{-a}))\bigr|\,dt\,d\delta<\infty$.
Hence we have
\begin{align*}
\PP\bigl(\sigma_{\{T_j\}}>c\bigr)
&=\frac1\pi\lim_{X\to\infty}
\Re\int_0^X\int_0^\infty\int_{x^{-\frac1a}}^\infty\exp\Bigl\{-itx-t^{\frac1a}\Phi_a(t\delta^{-a})\Bigr\}
\,d\delta\,dt\,dx
\\
&=\frac1\pi\lim_{X\to\infty}
\Re\int_0^\infty\int_0^X e^{-itx}\int_{x^{-\frac1a}}^\infty \exp\Bigl\{-t^{\frac1a}\Phi_a(t\delta^{-a})\Bigr\}
\,d\delta\,dx\,dt
\\
&=\frac1{\pi a}\lim_{X\to\infty}
\Re\int_0^\infty t^{\frac1a}\int_0^X e^{-itx}\int_0^{tx}e^{-t^{\frac1a}\Phi_a(y)}y^{-1-\frac1a}\,dy\,dx\,dt.
\end{align*}
Here, for any $t>0$, we have, by integration by parts:
\begin{align*}
\int_0^X e^{-itx}\int_0^{tx}e^{-t^{\frac1a}\Phi_a(y)}y^{-1-\frac1a}\,dy\,dx
\hspace{200pt}
\\
=\frac{ie^{-itX}}t\int_0^{tX}e^{-t^{\frac1a}\Phi_a(y)}y^{-1-\frac1a}\,dy
-\int_0^X ie^{-itx}e^{-t^{\frac1a}\Phi_a(tx)}(tx)^{-1-\frac1a}\,dx
\hspace{20pt}
\\
=\frac it\int_0^{tX}\bigl(e^{-itX}-e^{-iy}\bigr)e^{-t^{\frac1a}\Phi_a(y)}y^{-1-\frac1a}\,dy.
\end{align*}
Hence
\begin{align*}
\PP\bigl(\sigma_{\{T_j\}}>c\bigr)
=\frac1{\pi a}\lim_{X\to\infty}
\Im\int_0^\infty t^{\frac1a-1}\int_0^{tX} \bigl(e^{-iy}-e^{-itX}\bigr)e^{-t^{\frac1a}\Phi_a(y)}
y^{-1-\frac1a}\,dy\,dt.
\end{align*}

For given $X>1$, we split the integral over $t$ into two parts, corresponding to $t<X^{-1}$ and $t>X^{-1}$.
Regarding the first part, we note that $t<X^{-1}$ and $y<tX$ implies $y<1$.
Thus
$\Phi_a(y)=y^{-\frac1a}(1+O(y))$, and since also $t<X^{-1}<1$, we have
$e^{-t^{\frac1a}\Phi_a(y)}=e^{-(t/y)^{\frac1a}}(1+O(t^{\frac1a}y^{1-\frac1a}))$.
Recall that in this case we also have $e^{-iy}=1+O(y)$. Hence
\begin{align}\label{EXPLFORMULATHMPF5}
&\int_0^{X^{-1}} t^{\frac1a-1}\int_0^{tX} \bigl(e^{-iy}-e^{-itX}\bigr)e^{-t^{\frac1a}\Phi_a(y)}
y^{-1-\frac1a}\,dy\,dt
\\\notag
&=\int_0^{X^{-1}}\int_0^{tX} \Bigl(1-e^{-itX}+O\Bigl(y+t^{\frac1a}y^{1-\frac1a}\Bigr)\Bigr)
e^{-(t/y)^{\frac1a}} t^{\frac1a-1}y^{-1-\frac1a}\,dy\,dt
\\\notag
&=a\int_0^{X^{-1}}\int_{X^{-\frac1a}}^\infty \bigl(1-e^{-itX}+O\bigl(tu^{-a}+tu^{1-a}\bigr)\bigr)
e^{-u} t^{-1}\,du\,dt
\\\notag
&=a\int_0^1\frac{1-e^{-it}}t\,dt\int_{X^{-\frac1a}}^\infty e^{-u}\,du
+O\bigl(X^{-1}\bigr)\int_{X^{-\frac1a}}^\infty (u^{-a}+u^{1-a})e^{-u}\,du
\\\notag
&=a\int_0^1\frac{1-e^{-it}}t\,dt+O\bigl(X^{-\frac1a}\bigr).
\end{align}
The remaining part is
\begin{align}\label{EXPLFORMULATHMPF4}
&\int_{X^{-1}}^\infty t^{\frac1a-1}\int_0^{tX} \bigl(e^{-iy}-e^{-itX}\bigr)e^{-t^{\frac1a}\Phi_a(y)}
y^{-1-\frac1a}\,dy\,dt,
\end{align}
and here we have absolute convergence; %
$\int_{X^{-1}}^\infty t^{\frac1a-1}\int_0^{tX} e^{-t^{\frac1a}\Re\Phi_a(y)}y^{-1-\frac1a}\,dy\,dt<\infty$,
as is seen by a similar computation as in \eqref{EXPLFORMULATHMPF6}.
(The corresponding fact does not hold in \eqref{EXPLFORMULATHMPF5}.)
We also note that we may replace the range of the inner integral in \eqref{EXPLFORMULATHMPF4}
by all of $\R_{>0}$, to the cost of an error which is
\begin{align*}
\ll\int_{X^{-1}}^\infty t^{\frac1a-1}\int_{tX}^\infty e^{-\kappa_2t^{\frac1a}}y^{-1-\frac1a}\,dy\,dt
\ll X^{-\frac1a}\int_{X^{-1}}^\infty e^{-\kappa_2t^{\frac1a}}\,\frac{dt}t
\ll X^{-\frac1a}\log(2X).
\end{align*}

Collecting the above results, 
and using the fact that both $X^{-\frac1a}$ and $X^{-\frac1a}\log(2X)$ tend to zero as $X\to\infty$, we conclude that
\begin{align}\label{EXPLFORMULATHMPF7}
\PP\bigl(\sigma_{\{T_j\}}>c\bigr)
=\frac1{\pi}\int_0^1\frac{\sin t}t\,dt
+\frac1{\pi a}\lim_{X\to\infty}
\biggl(\int_{X^{-1}}^\infty\Im g_1(t)\,dt-\int_{X^{-1}}^\infty \Im\bigl(e^{-iXt}g_0(t)\bigr)\,dt\biggr),
\end{align}
where
\begin{align*}
g_\ell(t)=t^{\frac1a-1}\int_0^\infty e^{-i\ell y-t^{\frac1a}\Phi_a(y)}y^{-1-\frac1a}\,dy
\end{align*}
for $\ell=0,1$.

Next, we split $g_\ell(t)$ as $g_\ell(t)=g_{\ell,1}(t)+g_{\ell,2}(t)$, where
\begin{align*}
g_{\ell,1}(t)=t^{\frac1a-1}\int_0^1 e^{-i\ell y-t^{\frac1a}\Phi_a(y)}y^{-1-\frac1a}\,dy
\end{align*}
and
\begin{align*}
g_{\ell,2}(t)=t^{\frac1a-1}\int_1^\infty e^{-i\ell y-t^{\frac1a}\Phi_a(y)}y^{-1-\frac1a}\,dy.
\end{align*}
Bounding $\Re\Phi_a(y)$ from below as in \eqref{EXPLFORMULATHMPF6}, we see that for all $t>0$ we have
\begin{align}\label{EXPLFORMULATHMPF8}
\bigl|g_{\ell,1}(t)\bigr|\leq
t^{\frac1a-1}\int_0^1 \bigl|e^{-t^{\frac1a}\Phi_a(y)}\bigr|y^{-1-\frac1a}\,dy
\ll t^{-1}e^{-\kappa_4t^{\frac1a}},
\end{align}
where $\kappa_4$ is (just like $\kappa_1,\kappa_2,\kappa_3$) a positive number which may depend on $a$,
and
\begin{align}\label{EXPLFORMULATHMPF9}
\bigl|g_{\ell,2}(t)\bigr|\leq
t^{\frac1a-1}\int_1^\infty \bigl|e^{-t^{\frac1a}\Phi_a(y)}\bigr|y^{-1-\frac1a}\,dy
\ll t^{\frac1a-1}e^{-\kappa_2t^{\frac1a}}.
\end{align}
Note also that for all $t,y\in(0,1]$, we have
$e^{-iy-t^{\frac1a}\Phi_a(y)}=e^{-iy-t^{\frac1a}y^{-\frac1a}(1+O(y))}
=e^{-t^{\frac1a}y^{-\frac1a}}(1+O(t^{\frac1a}y^{1-\frac1a}+y))$, and thus
\begin{align}\notag
t^{\frac1a-1}\int_0^1\bigl|\Im e^{-iy-t^{\frac1a}\Phi_a(y)}\bigr|\,y^{-1-\frac1a}\,dy
\ll t^{\frac1a-1}\int_0^1\bigl(t^{\frac1a}y^{-\frac2a}+y^{-\frac1a}\bigr)e^{-t^{\frac1a}y^{-\frac1a}}\,dy
\hspace{30pt}
\\\label{EXPLFORMULATHMPF10}
\ll \int_{t^{\frac1a}}^\infty\bigl(v^{1-a}+v^{-a}\bigr)e^{-v}\,dv \ll t^{\frac1a-1}
\end{align}
for all $0<t\leq1$.
Combining this bound with \eqref{EXPLFORMULATHMPF8} and \eqref{EXPLFORMULATHMPF9}, we see that
\begin{align}\label{EXPLFORMULATHMPF12}
\int_0^\infty t^{\frac1a-1}\int_0^\infty\bigl|\Im e^{-iy-t^{\frac1a}\Phi_a(y)}\bigr|\,y^{-1-\frac1a}\,dy\,dt<\infty.
\end{align}
Hence the contribution from $g_1(t)$ in \eqref{EXPLFORMULATHMPF7} can be treated as follows:
\begin{align}\label{EXPLFORMULATHMPF13}
&\frac1{\pi a}\lim_{X\to\infty}\int_{X^{-1}}^\infty\Im g_1(t)\,dt
=\frac1{\pi a}\int_0^\infty t^{\frac1a-1}\int_0^\infty \Im e^{-iy-t^{\frac1a}\Phi_a(y)}y^{-1-\frac1a}\,dy\,dt
\\\notag
&=\frac1{\pi a}\int_0^\infty\Im\biggl(e^{-iy}\int_0^\infty t^{\frac1a-1} e^{-t^{\frac1a}\Phi_a(y)}\,dt\biggr)
y^{-1-\frac1a}\,dy
=\frac1{\pi}\int_0^\infty\Im\biggl(\frac{e^{-iy}}{\Phi_a(y)}\biggr)y^{-1-\frac1a}\,dy.
\end{align}

Finally, we treat the contribution from $g_0(t)$ in \eqref{EXPLFORMULATHMPF7}.
Note that, by \eqref{EXPLFORMULATHMPF8} and \eqref{EXPLFORMULATHMPF9}, the restriction of $g_0(t)$ to 
$[1,\infty)$ is an $\L^1$-function.
Hence, by the Riemann-Lebesgue lemma, 
$\int_1^\infty e^{-iXt}g_0(t)\,dt$ tends to $0$ as $X\to\infty$.
Moreover, the restriction of $g_{0,2}(t)$ to $(0,1]$ is in $\L^1$ and hence also
$\int_{X^{-1}}^1 e^{-iXt}g_{0,2}(t)\,dt$ tends to $0$ as $X\to\infty$.
Hence
\begin{align}\label{EXPLFORMULATHMPF11}
-\frac1{\pi a}\lim_{X\to\infty}\int_{X^{-1}}^\infty\Im\bigl(e^{-iXt}g_0(t)\bigr)\,dt
=-\frac1{\pi a}\lim_{X\to\infty}\Im\int_{X^{-1}}^1 e^{-iXt}g_{0,1}(t)\,dt.
\end{align}
Furthermore, for $0<t\leq1$, we have
\begin{align*}
g_{0,1}(t)=t^{\frac1a-1}\int_0^1 e^{-t^{\frac1a}\Phi_a(y)}y^{-1-\frac1a}\,dy
=t^{\frac1a-1}\int_0^1 e^{-t^{\frac1a}y^{-\frac1a}}\Bigl(1+O\Bigl(t^{\frac1a}y^{1-\frac1a}\Bigr)\Bigr)
y^{-1-\frac1a}\,dy
\\
=\frac at\int_{t^{\frac1a}}^\infty e^{-v}\,dv+O\bigl(t^{\frac1a-1}\bigr)
=\frac at+O\bigl(t^{\frac1a-1}\bigr),
\end{align*}
where we bounded the contribution from the big-$O$-term in the integral by a similar computation as in \eqref{EXPLFORMULATHMPF10}.
Thus $g_{0,1}(t)-\frac at$ is an $\L^1$-function on $t\in(0,1]$, so that
$\int_{X^{-1}}^1e^{-iXt}(g_{0,1}(t)-\frac at)\,dt$ tends to $0$ as $X\to\infty$.
Hence \eqref{EXPLFORMULATHMPF11} equals
\begin{align*}
-\frac1{\pi}\lim_{X\to\infty}\Im\int_{X^{-1}}^1 \frac{e^{-iXt}}t\,dt
=\frac1\pi\lim_{X\to\infty}\int_1^X \frac{\sin t}t\,dt
=\frac1\pi\int_1^\infty \frac{\sin t}t\,dt.
\end{align*}

Collecting our results into \eqref{EXPLFORMULATHMPF7}, we obtain,
since $\int_0^\infty\frac{\sin t}t\,dt=\frac\pi2$,
\begin{align}\label{EXPLFORMULATHMPFFINAL}
\PP\bigl(\sigma_{\{T_j\}}>c\bigr)
=\frac12+\frac1{\pi}\int_0^\infty\Im\biggl(\frac{e^{-iy}}{\Phi_a(y)}\biggr)y^{-1-\frac1a}\,dy.
\end{align}
Let us note that $\Phi_a(y)$ can be expressed in terms of the incomplete gamma function, by
substituting $u=iv$ in \eqref{PHIAYDEF} and using formulas \eqref{INCOMPLETEGAMMADEF} and \eqref{INCOMPLETEGAMMAREC}:
\begin{align}\label{PHIAYFORMULA}
\Phi_a(y)=y^{-\frac1a}e^{iy}+e^{-\frac\pi{2a}i}\gamma\Bigl(1-\frac1a,-iy\Bigr)
=-\frac{e^{-\frac\pi{2a}i}}{a}\gamma\Bigl(-\frac1a,-iy\Bigr).
\end{align}
Substituting this into \eqref{EXPLFORMULATHMPFFINAL}, we obtain the formula stated in
Theorem \ref{EXPLFORMULATHM}.
Using $|\Phi_a(y)|\geq\Re\Phi_a(y)\geq\kappa_2>0$ for all $y>0$ and $\Phi_a(y)=y^{-\frac1a}(1+O(y))$ for
$0<y\leq1$, one immediately sees that the integral in \eqref{EXPLFORMULATHMPFFINAL} is absolutely convergent
(this is also clear from the proof, cf.\ in particular \eqref{EXPLFORMULATHMPF12} and \eqref{EXPLFORMULATHMPF13}).
This concludes the proof of Theorem \ref{EXPLFORMULATHM}.
\end{proof}

\begin{remark}
It is worth stressing that if we remove the imaginary part in \eqref{EXPLFORMULATHMPFFINAL},
then convergence \textit{fails}:
We have $\bigl|\int_{y_0}^1\frac{e^{-iy}}{\Phi_a(y)}y^{-1-\frac1a}\,dy\bigr|\to\infty$ as $y_0\to0^+$,
since $\Phi_a(y)=y^{-\frac1a}(1+O(y))$ for $0<y\leq1$.
\end{remark}

\section{Proof of Corollary \ref{EXPLFORMULACOR}}
\label{EXPLFORMULACORPFSEC}

In this section we prove Corollary \ref{EXPLFORMULACOR}.
To begin, note that by  formal differentiation under the integral sign in \eqref{EXPLFORMULATHMPFFINAL}, we have
$Prob\bigl(\sigma_{\{T_j\}}\leq c\bigr)=\int_{1/2}^c f(c_1)\,dc_1$, where
$f:\R_{>\frac12}\to\R_{>0}$ is given by
\begin{align}\label{DENSITY}
f(c) %
=\frac2{\pi}\int_0^\infty\Im\left(\frac{e^{-iy}}{\Phi_a(y)}\left(\frac{\frac{\partial}{\partial a}\Phi_a(y)}{\Phi_a(y)}-\frac{\log y}{a^2}\right)\right)y^{-1-\frac1a}\,dy.
\end{align}
Here $a:=2c\in\R_{>1}$ (see Section \ref{EXPLFORMULATHMSEC}).
This manipulation is justified by the fact that the integrand in \eqref{DENSITY} is majorized, uniformly for $a$ in compact subsets of $\R_{>1}$, by an integrable function; this follows from an argument similar to the one that shows that the integral in \eqref{EXPLFORMULATHMPFFINAL} is absolutely convergent, using also that $\frac{\partial}{\partial a}\Phi_a(y)=a^{-2}(\log y)y^{-\frac1a}\left(1+O(y)\right)$ for $0<y\leq\frac12$ and $\frac{\partial}{\partial a}\Phi_a(y)=O(1)$ for $\frac12\leq y<\infty$.

\begin{remark}
Note in particular that %
the imaginary part in \eqref{DENSITY} may be taken outside the integral; in fact even 
$\int_0^\infty \bigl|\frac{e^{-iy}}{\Phi_a(y)}\bigl(\frac{\frac{\partial}{\partial a}\Phi_a(y)}{\Phi_a(y)}-\frac{\log y}{a^2}\bigr)\bigr|y^{-1-\frac1a}\,dy<\infty$.
\end{remark}

Let us now consider formula \eqref{DENSITY} in the limit as $a\to\infty$.
In \eqref{PHIAYDEF}, we expand $e^{iu}$ in a power series, change order between summation and integration and then use
$(n-a^{-1})^{-1}=n^{-1}\sum_{k=0}^\infty(na)^{-k}$ for each $n\in\Z^+$.
This gives
\begin{align}\label{ALARGEPF8}
\Phi_a(y)=y^{-\frac1a}\Bigl(1-\sum_{k=1}^\infty F_k(y)a^{-k}\Bigr),
\end{align}
where
\begin{align}\label{FKYDEF}
F_k(y):=\sum_{n=1}^\infty\frac{(iy)^n}{n!n^k}.
\end{align}
Obviously $|F_k(y)|\leq e^{|y|}-1$ holds for all $y>0$ and all $k$,
and hence we see that given any $y_0>0$ there exists some $a_0=a_0(y_0)>1$ such that
$\bigl|\sum_{k=1}^\infty F_k(y)a^{-k}\bigr|\leq\frac12$ holds for all $a\geq a_0$, $y\in(0,y_0]$.
We also have $\bigl|\sum_{k=1}^\infty F_k(y)a^{-k}\bigr|\ll|y|a^{-1}$ for these $a,y$, and therefore
$\Phi_a(y)^{-1}=y^{\frac1a}(1+F_1(y)a^{-1}+O(|y|a^{-2}))$.
The power series in \eqref{ALARGEPF8} may also be differentiated termwise with respect to $a$.
Using these observations, we obtain by a short calculation:
\begin{align}\label{ALARGEPF10}
\frac{y^{-1-\frac1a}}{\Phi_a(y)}\biggl(\frac{\frac{\partial}{\partial a}\Phi_a(y)}{\Phi_a(y)}-\frac{\log y}{a^2}\biggr)
=\frac{F_1(y)}{ya^2}+\frac{2(F_1(y)^2+F_2(y))}{ya^3}+O\Bigl(\frac{1+|\log y|}{a^4}\Bigr),
\end{align}
uniformly over all $a\geq a_0(y_0)$, $y\in(0,y_0]$ (where we recall that $y_0>0$ is arbitrary).

In order to obtain a similar relation also for large $y$, we start by setting
\begin{align}\label{ALARGEPF3xi}
\xi(a):=\lim_{y\to\infty}\Phi_a(y)
=-\frac{e^{-\frac\pi{2a}i}\Gamma(-\frac1a)}a %
\end{align}
(cf.\ \eqref{PHIAYFORMULA}). In view of \eqref{PHIAYDER}, we have 
$\Phi_a(y)=\xi(a)+a^{-1}\int_y^\infty u^{-1-\frac1a}e^{iu}\,du$, and integrating by parts twice, we get 
(for any $a>1$, $y>0$)
\begin{align}\label{ALARGEPF3}
\Phi_a(y)=\xi(a)+\frac{y^{-\frac1a}}a\Gamma(0,-iy)
-\frac{y^{-\frac1a}}{a^2}\Pi(y)+\frac1{a^3}\int_y^\infty u^{-1-\frac1a}\Pi(u)\,du,
\end{align}
where
\begin{align}\label{MYPIDEF}
\Pi(y):=\int_y^\infty\frac{\Gamma(0,-iu)}u\,du.
\end{align}
We have $|\Gamma(0,-iy)|\ll y^{-1}$ for all $y>0$, and thus also $|\Pi(y)|\ll y^{-1}$ for $y\geq 1$. %
Using this fact together with the trivial observation $-1-\frac1a<-1$, we bound the integral in \eqref{ALARGEPF3} and get
\begin{align}\label{ALARGEPF1}
\Phi_a(y)=\xi(a)+\frac{y^{-\frac1a}}a\Gamma(0,-iy)
-\frac{y^{-\frac1a}}{a^2}\Pi(y)+O(a^{-3}y^{-1}),
\end{align}
uniformly over all $a>1$, $y\geq1$.
Since also $\xi(a)=1+(\gamma-\tfrac{\pi}2i)a^{-1}+O(a^{-2})$ as $a\to\infty$, we see that
$\Phi_a(y)/\xi(a)$ is near $1$ whenever $a$ and $y$ are %
large; hence there exist absolute constants $a_1>1$ and $y_0\geq1$ such that
for all $a\geq a_1$ and $y\geq y_0$,
\begin{align}\notag
\frac1{\Phi_a(y)} &=\frac1{\xi(a)}-\frac{y^{-\frac1a}}{a\xi(a)^2}\Gamma(0,-iy)+O(a^{-2}y^{-1})
\\\label{ALARGEPF2}
&=\frac1{\xi(a)}-\frac{y^{-\frac1a}}{a}\Gamma(0,-iy)+O(a^{-2}y^{-1}).
\end{align}
In order to obtain an asymptotic formula also for $\frac{\partial}{\partial a}\Phi_a(y)$, we note that the right-hand side
of \eqref{ALARGEPF3} defines an analytic function of the \emph{complex} variable $w=a^{-1}$ in the region $|w|<1$ (including $w=0$).
Restricting to $|w|\leq\frac12$, %
we may bound the absolute value of the integral in \eqref{ALARGEPF3}
using $|\Pi(u)|\ll u^{-1}$ and $\Re(-1-w)\leq-\frac12$.
We may then use the Cauchy differentiation formula to obtain an asymptotic formula for the
derivative of our analytic function, valid uniformly for $|w|\leq\frac14$.
In particular,
\begin{align}\label{ALARGEPF5}
\frac{\partial}{\partial a}\Phi_a(y)=\xi'(a)-\frac{y^{-\frac1a}}{a^2}\Gamma(0,-iy)
+\frac{y^{-\frac1a}}{a^3}\bigl(2\Pi(y)+\Gamma(0,-iy)\log y\bigr)
+O\bigl(a^{-4}y^{-\frac12}\bigr),
\end{align}
uniformly over all $a\geq4$ and all $y\geq1$.
Using \eqref{ALARGEPF2} and \eqref{ALARGEPF5}, we obtain, via a straightforward computation,
\begin{align}\label{ALARGEPF6}
\frac{y^{-1-\frac1a}}{\Phi_a(y)}\biggl(\frac{\frac{\partial}{\partial a}\Phi_a(y)}{\Phi_a(y)}-\frac{\log y}{a^2}\biggr)
=\frac{\xi'(a)}{\xi(a)^2}y^{-1-\frac1a}-\frac{(\log y)y^{-1-\frac1a}}{a^2\xi(a)}
-\frac{y^{-1-\frac2a}\Gamma(0,-iy)}{a^2}
\hspace{30pt}
\\\notag
+\frac{y^{-1-\frac2a}}{a^3}
\biggl\{\bigl(4\gamma-2\pi i+2\log y+2y^{-\frac1a}\Gamma(0,-iy)\bigr)\Gamma(0,iy)+2\Pi(y)\biggr\}
+O\bigl(a^{-4}y^{-\frac32}\bigr),
\end{align}
uniformly over $a\geq\max(a_1,4)$ and $y\geq y_0$.

We now multiply the relation \eqref{ALARGEPF6} with $e^{-iy}$, and integrate the result over $y\in[y_0,\infty)$.
The contribution from the first term is  %
\begin{align}\label{ALARGEPF7}
\frac{\xi'(a)}{\xi(a)^2}\int_{y_0}^\infty y^{-1-\frac1a}e^{-iy}\,dy.
\end{align}
We split this integral into two parts as $\int_{y_0}^{\exp(a^{1/4})}+\int_{\exp(a^{1/4})}^\infty$
(keeping $a$ so large that $\exp(a^{1/4})>y_0$); then, because of the oscillating character of the integrand,
the second integral is $O(\exp(-a^{1/4}))$.
In the first integral, we use $y^{-\frac1a}=1-\frac{\log y}a+\frac12(\frac{\log y}a)^2+O((\frac{\log y}a)^3)$
and
$\int_{y_0}^{\exp(a^{1/4})}\frac{(\log y)^3}y\,dy\ll a$; then, by a quick computation, we find that \eqref{ALARGEPF7} equals
$\frac{\xi'(a)}{\xi(a)^2}
\bigl(\int_{y_0}^\infty\frac{e^{-iy}}y\,dy-\frac1a\int_{y_0}^\infty\frac{e^{-iy}\log y}y\,dy+O(a^{-2})\bigr)$.
The remaining terms in \eqref{ALARGEPF6} can be treated similarly, and using the relations
\begin{align}\label{F1F2FORMULAS}
&F_1(y)=\tfrac\pi2i-\gamma-\log y-\Gamma(0,-iy);
\\\notag
&F_2(y)=\Pi(y)+\bigl(\tfrac1{24}\pi^2-\tfrac12\gamma^2+\tfrac12\pi i\gamma\bigr)
+\bigl(\tfrac12\pi i-\gamma-\tfrac12\log y\bigr)\log y,
\end{align}
the result may be collected as
\begin{align}\notag
&\int_{y_0}^\infty
\frac{e^{-iy}y^{-1-\frac1a}}{\Phi_a(y)}\biggl(\frac{\frac{\partial}{\partial a}\Phi_a(y)}{\Phi_a(y)}-\frac{\log y}{a^2}\biggr)
\,dy
\hspace{140pt}
\\\notag %
&=a^{-2}\int_{y_0}^\infty\frac{F_1(y)}y e^{-iy}\,dy
+a^{-3}\int_{y_0}^\infty\frac{2(F_1(y)^2+F_2(y))}y e^{-iy}\,dy+O(a^{-4}),
\end{align}
for all $a\geq\max(a_1,4)$.
Using also \eqref{DENSITY} and \eqref{ALARGEPF10}, we thus obtain an asymptotic formula for $f(c)$ as $c=\tfrac12a\to\infty$.
Note, however, that $\Im\int_0^\infty\frac{F_1(y)}ye^{-iy}\,dy=\frac12\Im\int_{-\infty}^\infty\frac{F_1(y)}ye^{-iy}\,dy=0$, 
where the second equality follows using the Cauchy integral theorem, moving the contour towards infinity in the 
lower half-plane.
Hence the coefficient in front of $a^{-2}=(2c)^{-2}$ in the asymptotic formula vanishes, and 
we arrive at \eqref{ASYMPTCLARGERES}, with
\begin{align}\label{KINT}
K_2=\frac1{2\pi}\Im\int_{0}^{\infty}\frac{F_1(y)^2+F_2(y)}y e^{-iy}\,dy=0.822467\ldots.
\end{align}
(The numerical evaluation of this integral, which is not entirely straightforward, is carried out in 
\cite[constants.mpl]{sts}.)

\vspace{5pt}

We next turn to the study of \eqref{DENSITY} in the limit as $a\to1$.
Our presentation here will be rather brief; we refer to \cite[asymptotics.mpl]{sts} for further details. 
The formula \eqref{PHIAYDEF} may be expressed as
\begin{align}\label{ASMALLPF2}
\Phi_a(y)=y^{-\frac1a}e^{iy}-ie^{iy}\frac{y^{1-\frac1a}}{1-a^{-1}}-\frac 1{1-a^{-1}}\int_0^y e^{iu}u^{1-\frac1a}\,du.
\end{align}
Now fix $N\in\Z^+$, and let us keep $(a-1)^N\leq y\leq(a-1)^{-N}$, and $a\in(1,2]$.
We split the integral in \eqref{ASMALLPF2} as $\int_0^{(a-1)^{N+1}}+\int_{(a-1)^{N+1}}^y$ and bound the first part trivially,
while for $u\in[(a-1)^{N+1},y]$, we use the fact that 
$u^{1-\frac1a}=\sum_{k=0}^{N+1}\frac{(1-a^{-1})^k(\log u)^k}{k!}$\linebreak
$+O((a-1)^{N+2}|\log u|^{N+2})$, where the error is an increasing function of $u$ when $u\geq1$.
This leads to the formula
\begin{align}\label{ASMALLPF1}
&\Phi_a(y)=\frac{-i}{1-a^{-1}}\biggl\{1+\sum_{k=1}^{N+1} G_k(y)(1-a^{-1})^k
+O\Bigl((a-1)^{N+1}\Bigl(1+\frac1y\Bigr)\Bigr)\biggr\},
\end{align}
where $G_1(y)$, $G_2(y),\ldots$ are given by
\begin{align*}
G_k(y):=\frac{ie^{iy}(\log y)^{k-1}}{(k-1)! y}+\frac{e^{iy}(\log y)^k}{k!}
-\frac i{k!}\int_0^y(\log u)^ke^{iu}\,du.
\end{align*}
Let us now further restrict to the case where $(a-1)^{\frac12}\leq y\leq (a-1)^{-N}$.
Using $|G_k(y)|\ll_k |\log y|^{k-1}y^{-1}+|\log y|^k+1$, we see that there is some $a_0=a_0(N)\in(1,2]$ such that
for all $a\in(1,a_0]$ and all $y\in[(a-1)^{\frac12},(a-1)^{-N}]$
the expression within the brackets in \eqref{ASMALLPF1} lies in $\{z\col |z-1|<\frac12\}$, and so we get
\begin{align*}
\frac1{\Phi_a(y)}
=i(1-a^{-1})\biggl\{1+\sum_{\ell=1}^N\biggl\{-\sum_{k=1}^N G_k(y)(1-a^{-1})^k\biggr\}^\ell
\hspace{100pt}
\\
+O\Bigl(\Bigl(y^{-1}+|\log y|\Bigr)^{N+1}(a-1)^{N+1}\Bigr)
\biggr\}.
\end{align*}
Working similarly, starting from a differentiated version of \eqref{ASMALLPF2}, we also get an asymptotic formula for
$\frac{\partial}{\partial a}\Phi_a(y)$,
and with further computation, we finally obtain
\begin{align*}
\frac{e^{-iy}y^{-1-\frac1a}}{\Phi_a(y)}
\biggl(\frac{\frac{\partial}{\partial a}\Phi_a(y)}{\Phi_a(y)}-\frac{\log y}{a^2}\biggr)
=-\frac{ie^{-iy}}{y^2}\biggl\{1+\sum_{\ell=1}^N H_\ell(y)(a-1)^\ell
\hspace{100pt}
\\
+O\Bigl(\Bigl(y^{-1}+|\log y|\Bigr)^{N+1}(a-1)^{N+1}\Bigr)\biggr\},
\end{align*}
for all $a\in(1,a_0]$ and $y\in[(a-1)^{\frac12},(a-1)^{-N}]$.
Here $H_1(y),H_2(y),\ldots$ are certain continuous functions of $y$ satisfying
$|H_\ell(y)|\ll_\ell (y^{-1}+|\log y|)^\ell$; in particular
we have
\begin{align}\label{HFORMULA}%
&H_1(y)=2\biggl\{i\int_0^y(\log u)e^{iu}\,du-\frac{ie^{iy}}y+(1-e^{iy})\log y-1\biggr\};
\\\notag
&H_2(y)=3\biggl\{\frac i2\int_0^y(\log u)^2e^{iu}\,du
-\biggl(\int_0^y(\log u)e^{iu}\,du+i-\frac i2\log y-\frac{e^{iy}}y+ie^{iy}\log y\biggr)^2
\\\notag
&\hspace{200pt}-\frac{ie^{iy}\log y}y-\log y+\bigl(\tfrac14-\tfrac12e^{iy}\bigr)(\log y)^2\biggr\}.
\end{align}
Writing $\widetilde H_\ell(y):=-ie^{-iy}y^{-2}H_\ell(y)$, it follows that, for $y\leq1$,
\begin{align*}
&\Im\widetilde H_1(y)=2y^{-2}+\tfrac12-\tfrac{13}{144}y^2+O(y^4);
\qquad \Im\widetilde H_2(y)=3y^{-4}+\tfrac32y^{-2}-\tfrac{17}6+O(y^2).
\end{align*}
Furthermore, one computes (again for $y\leq1$)
\begin{align*}
&\Im\widetilde H_3(y)=-4y^{-4}-\tfrac{37}3y^{-2}+O(1);
\quad&&
\Im\widetilde H_4(y)=-5y^{-6}-\tfrac{15}2y^{-4}+O(y^{-2});
\\
&\Im\widetilde H_5(y)=6y^{-6}+O(y^{-4}).
\quad&&
\Im\widetilde H_6(y)=7y^{-8}+O(y^{-6}).
\end{align*}
Using these relations (taking $N=6$), we obtain
\begin{align}\notag
\Im\int_{(a-1)^{\frac12}}^\infty %
\frac{e^{-iy}y^{-1-\frac1a}}{\Phi_a(y)}
\biggl(\frac{\frac{\partial}{\partial a}\Phi_a(y)}{\Phi_a(y)}-\frac{\log y}{a^2}\biggr)\,dy
=\int_0^\infty\frac{1-\cos y}{y^2}\,dy
\hspace{100pt}
\\\notag
+(a-1)\int_0^\infty\Bigl(\Im\widetilde H_1(y)-2y^{-2}\Bigr)\,dy
+(a-1)^2\int_0^\infty\Bigl(\Im\widetilde H_2(y)-3y^{-4}-\tfrac32y^{-2}\Bigr)\,dy
\\\label{ASMALLPF3}
+\Bigl\{-(a-1)^{-\frac12}+\tfrac52(a-1)^{\frac12}-\tfrac{95}{72}(a-1)^{\frac32}-\tfrac{52759}{5400}(a-1)^{\frac52}\Bigr\}+O\bigl((a-1)^3\bigr).
\end{align}
(This formula is first derived with each upper integration limit being $(a-1)^{-4}$ (say) in place of $\infty$;
the remaining integrals over $y\in[(a-1)^{-4},\infty)$ are easily seen to be subsumed in the error %
term.)

To treat the integral over $y\leq(a-1)^{\frac12}$, we start with the formula
\begin{align*}
\Phi_a(y)=\frac{y^{-\frac1a}(a-1-iy)}{a-1}\biggl\{1-\sum_{k=2}^N\frac{i^k(a-1)}{k!(ka-1)(a-1-iy)}y^k+
O\Bigl(y^N\min(a-1,y)\Bigr)\biggr\},
\end{align*}
which holds uniformly over all $a>1$ and $0<y\leq1$, for any fixed $N\in\Z_{\geq2}$;
this is proved using \eqref{PHIAYDEF} and the power series expansion of $e^{iu}$.
Note that the sum over $k$ is $O(y\min(a-1,y))$; hence there is an absolute constant $y_0\in(0,1]$ such that
for all $a>1$ and $0<y\leq y_0$, we have
\begin{align*}
\frac1{\Phi_a(y)}=\frac{y^{\frac1a}(a-1)}{a-1-iy}\biggl\{1+\sum_{1\leq\ell\leq N/2}
\biggl(\sum_{k=2}^N\frac{i^k(a-1)}{k!(ka-1)(a-1-iy)}y^k\biggr)^\ell+
O\Bigl(y^N\min(a-1,y)\Bigr)\biggr\}.
\end{align*}
Using this formula with $N=5$, together with a similar asymptotic formula for $\frac{\partial}{\partial a}\Phi_a(y)$ 
deduced from a differentiated version of \eqref{PHIAYDEF},
we find after some computation that
\begin{align*}
\frac{e^{-iy}y^{-1-\frac1a}}{\Phi_a(y)}
\biggl(\frac{\frac{\partial}{\partial a}\Phi_a(y)}{\Phi_a(y)}-\frac{\log y}{a^2}\biggr)
=\frac{P_0(a-1,y)+P_1(a-1,y)\log y}{a^2(2a-1)^6(3a-1)^5(4a-1)^4(5a-1)^4(a-1-iy)^6}
\\
+O\Bigl((a-1)y^3\bigl(1+(a-1)|\log y|\bigr)\Bigr),
\end{align*}
where $P_0$ and $P_1$ are explicit polynomials.
This formula can now be integrated over $y$ in terms of elementary functions, and we obtain
\begin{align}\notag
\Im\int_0^{(a-1)^{\frac12}}
\frac{e^{-iy}y^{-1-\frac1a}}{\Phi_a(y)}
\biggl(\frac{\frac{\partial}{\partial a}\Phi_a(y)}{\Phi_a(y)}-\frac{\log y}{a^2}\biggr)\,dy
=\tfrac12\pi-\tfrac52\pi(a-1)^2
\hspace{70pt}
\\\label{ASMALLPF4}
+\Bigl\{(a-1)^{-\frac12}-\tfrac52(a-1)^{\frac12}+\tfrac{95}{72}(a-1)^{\frac32}+\tfrac{52759}{5400}(a-1)^{\frac52}\Bigr\}+O\bigl((a-1)^3\bigr).
\end{align}
Finally, we add \eqref{ASMALLPF3} and \eqref{ASMALLPF4}, %
and note that since $H_1(y)=-2(F_1(y)+\frac{ie^{iy}}y+1)$, we have
\begin{align*}
\int_0^\infty\bigl(\Im\widetilde H_1(y)-2y^{-2}\bigr)\,dy
=\int_{-\infty}^\infty\frac{\Im\big(ie^{-iy}F_1(y)\big)}{y^2}\,dy-\pi=0,
\end{align*}
where the second equality follows by 
again moving the contour towards infinity in the lower half-plane, noticing the pole at $y=0$.
Hence we arrive at \eqref{ASYMPTCSMALLRES}, with
\begin{align}\label{K1DEF}
K_1=20+\frac8\pi\int_0^\infty\Bigl(3y^{-4}+\tfrac32y^{-2}-\Im\widetilde H_2(y)\Bigr)\,dy=39.47841\ldots
\end{align}
(cf.\ \cite[constants.mpl]{sts}).
This completes the proof of Corollary \ref{EXPLFORMULACOR}.
\hfill$\square$

\begin{remark}
It appears %
that by the same method %
one could obtain asymptotic expansions of %
$f(c)$, in the limits as $c\to\infty$ and $c\to\frac12$,
with the error term having any desired power rate of decay. %
\end{remark}

\appendix
\section{Residue calculus and numerical computation\\ of the density}
\label{NUMSEC}

In this appendix we  discuss the evaluation of the
integrals in \eqref{EXPLFORMULATHMPFFINAL} and \eqref{DENSITY} using the residue theorem,
resulting in alternative formulas for $\PP\bigl(\sigma_{\{T_j\}}>c\bigr)$ and the corresponding density.
These formulas turn out to be useful for numerical computation,
something which we discuss briefly towards the end of the appendix (see also \cite[numdensity.mpl]{sts}).

We now write $z$ in place of $y$.
By \eqref{PHIAYDEF} we have
$\Phi_a(z)=z^{-\frac1a}(e^{iz}-iz\int_0^1 e^{izt}t^{-\frac1a}\,dt)$,
and here the expression in the parenthesis is clearly an entire function of $z$.
Hence
\begin{align}\label{PSIADEF}
\Psi_a(z):=\frac{e^{-iz}z^{-1-\frac1a}}{\Phi_a(z)}=\frac{e^{-iz}z^{-1}}{e^{iz}-iz\int_0^1 e^{izt}t^{-\frac1a}\,dt}
\end{align}
is a meromorphic function in all of $\C$.
In \eqref{EXPLFORMULATHMPFFINAL} we are integrating $\Im\Psi_a(z)$ along the positive real line;
using the symmetry $\Psi_a(-z)=-\overline{\Psi_a(\overline z)}$, we may rewrite this as
\begin{align}\label{EXPLFORMULATHMPFFINAL3}
\PP\bigl(\sigma_{\{T_j\}}>c\bigr)
=\frac12+\lim_{r\to0^+}\frac1{2\pi i}\Bigl(\int_{-\infty}^{-r}\Psi_a(y)\,dy+\int_r^{\infty}\Psi_a(y)\,dy\Bigr).
\end{align}
Let $C'_r$ be the semicircle $\{z\col |z|=r,\:\Im z\leq0\}$, oriented in the direction from $-r$ to $r$,
and let $C_r$ be the contour going from $-\infty$ to $-r$ along $\R$, then from $-r$ to $r$ along $C'_r$
and finally from $r$ to $+\infty$ along $\R$.
Since $\Psi_a(z)$ has a simple pole at $z=0$ with residue $1$, we have
$\int_{C'_r}\Psi_a(z)\,dz=i\pi+O(r)$ as $r\to0$.
Thus \eqref{EXPLFORMULATHMPFFINAL3} equals $\lim_{r\to0^+}\frac1{2\pi i}\int_{C_r}\Psi_a(z)\,dz$.
However, by Cauchy's integral theorem, $\int_{C_r}\Psi_a(z)\,dz$ is independent of $r$ for all sufficiently small $r$.
Hence
\begin{align}\label{CONTOURCHANGE1}
\PP\bigl(\sigma_{\{T_j\}}>c\bigr)=\frac1{2\pi i}\int_{C_r}\Psi_a(z)\,dz,
\end{align}
for any $r>0$ so small that $\Psi_a(z)$ has no pole in the punctured disk $\{z\col 0<|z|\leq r\}$.

We wish to replace $C_r$ in \eqref{CONTOURCHANGE1} by a contour over $z$'s with large negative imaginary part.
In order to do so, we first need to understand the poles of $\Psi_a(z)$ in the lower half plane.
Numerics indicate that there is exactly one simple pole in the infinite vertical strip
$\{z\col (2n-1)\pi<\Re z<(2n+1)\pi,\:\Im z<0\}$ for each integer $n$; cf.\ Figure \ref{ZEROTRACKINGFIG} below.
However, for technical reasons it seems easier to prove a corresponding statement instead for certain
``curved vertical strips'', as follows.
For each $n\in\Z^+$, we let $\Gamma_n$ be the curve in the complex plane given by
\begin{align}\label{CURVEDEF}
x\mapsto c_n(x)=x-ix\tan\bigl((n-\tfrac14)\pi-\tfrac12x\bigr),
\qquad \bigl(2n-\tfrac32\bigr)\pi<x\leq\bigl(2n-\tfrac12\bigr)\pi.
\end{align}
One notes that $\Im c_n(x)\to-\infty$ as $x\to(2n-\frac32)\pi^+$, that $\Im c_n((2n-\frac12)\pi)=0$ and that $0<\arg c_n'(x)<\frac\pi2$ for all $(2n-\tfrac32)\pi<x<(2n-\tfrac12)\pi$. Hence $\Gamma_n$ and $\Gamma_{n+1}$, together with the real interval $[(2n-\tfrac12)\pi,(2n+\tfrac32)\pi]$,
bound a curved vertical strip, which we call $S_n$ (we take $S_n$ to be closed).
We also let $S_{-n}=\{-\overline z\col z\in S_n\}$ be the reflection of $S_n$ in the imaginary axis,
and we let $S_0$ be the curved vertical strip bounded by the curves $\Gamma_1$,
$\{-\overline z\col z\in\Gamma_1\}$ and $[-\frac32\pi,\frac32\pi]$.
Now the union of all $S_n$ ($n\in\Z$) equals the negative half plane, $\{z\col\Im z\leq0\}$,
and the $S_n$'s have pairwise disjoint interiors.
\begin{prop}\label{UNIQUEPOLELEM}
Let $a>1$ be given.
For each $n\in\Z$, the function $z\Psi_a(z)$ has a unique pole in the strip $S_n$.
This pole is simple, and lies in the interior of $S_n$.
\end{prop}

For the proof we need the following lemma.
We will use the definition \eqref{INCGAMMADEF} of $\Gamma(s,z)$ for general $z\in\C\setminus\R_{\leq0}$,
the integral %
being over %
the infinite ray $u\in z+\R_{>0}$.
\begin{lem}\label{SPECINCOMPLGAMMABOUNDLEM}
For any $s\in[1,2]$ and any $z=-x+iy\in\C$, satisfying either 
$\frac12\pi\leq|y|\leq\frac12x$, $\frac34\pi\leq|y|\leq x$ or $[x\geq0$ and $|y|\geq\pi]$,
we have
\begin{align}\label{SPECINCOMPLGAMMABOUNDLEMRES}
\bigl|\Gamma(-s,z)\bigr|<s^{-1}|z|^{-s}e^x.
\end{align}
\end{lem}
\begin{proof}
Take $s$ and $z=-x+iy$ satisfying the assumptions. %
By symmetry, we may assume $y>0$. %
We may deform the contour of integration in \eqref{INCGAMMADEF} %
to be the ray
$\{z+t(1+ki)\col t\geq0\}$, where $k$ is any fixed non-negative number.
This ray intersects the imaginary axis at $(y+kx)i$, and thus
$|u|\geq (y+kx)(1+k^2)^{-1/2}$ holds for every point $u$ on the ray, and
\begin{align}\label{SPECINCOMPLGAMMABOUNDLEMPF1}
\bigl|\Gamma(-s,z)\bigr| 
\leq\frac{(1+k^2)^{\frac{s+1}2}}{(y+kx)^{s+1}}\int_0^\infty e^{x-t}(1+k^2)^{\frac12}\,dt
=\frac{(1+k^2)^{1+\frac s2}}{(y+kx)^{s+1}}e^x.
\end{align}
Applying this with $k=1$,
we see that \eqref{SPECINCOMPLGAMMABOUNDLEMRES} holds whenever
$s\bigl(\frac{\sqrt2 |z|}{x+y}\bigr)^s<\frac{x+y}2$.
But $\frac{\sqrt2 |z|}{x+y}\geq1$ for all non-zero $z$ and thus the inequality
holds for all $s\in[1,2]$ if and only if it holds for $s=2$,
i.e.\ if and only if $\frac{x^2+y^2}{(x+y)^3}<\frac18$.
However, it is easily verified %
that $\frac{x^2+y^2}{(x+y)^3}$ is a decreasing function of $x>0$ for any fixed $y\geq0$.
Hence, if $x\geq y\geq\frac34\pi$, then
$\frac{x^2+y^2}{(x+y)^3}\leq\frac1{4y}\leq\frac1{4\cdot\frac34\pi}<\frac18$;
similarly, if $x\geq2y\geq\pi$, then
$\frac{x^2+y^2}{(x+y)^3}\leq\frac5{27\cdot\frac12\pi}<\frac18$,
and if $y\geq\pi$ and $x\geq\frac34y$, then
$\frac{x^2+y^2}{(x+y)^3}\leq %
\frac{100}{343y}\leq\frac{100}{343\pi}<\frac18$.
To treat the remaining case, when $y\geq\pi$ and $0\leq x<\frac34y$,
we apply \eqref{SPECINCOMPLGAMMABOUNDLEMPF1} with $k=0$;
from this we see that \eqref{SPECINCOMPLGAMMABOUNDLEMRES} holds whenever $s(|z|/y)^s<y$.
However, if $y\geq\pi$ and $0\leq x<\frac34y$, then
$s(|z|/y)^s\leq2(|z|/y)^2<\frac{25}8<\pi\leq y$, and we are done.
\end{proof}

We also record the following bound, which follows from \eqref{SPECINCOMPLGAMMABOUNDLEMPF1} with $k=1$:
\begin{lem}\label{INCPGAMMAEASYLEM}
The bound $\bigl|\Gamma(-s,z)\bigr|\ll |z|^{-s-1}e^{-\Re z}$ holds uniformly for all
$s\in[1,2]$ and all $z\in\C$ with $\Re z\leq0$, $\Im z\neq0$.
\end{lem}

\begin{proof}[Proof of Proposition \ref{UNIQUEPOLELEM}]
Let $\eta_a(z)=z^{\frac1a}\Phi_a(z)=e^{iz}-iz\int_0^1 e^{izt}t^{-\frac1a}\,dt$ and note that $\eta_a$ is an entire function.
By \eqref{PSIADEF}, our task is to prove that for each $n$,
$\eta_a(z)$ has a unique zero in $S_n$, which is %
simple and lies in the interior of $S_n$.
Using \eqref{PHIAYFORMULA} and applying the recursion formula
$\Gamma(s,z)=e^{-z}z^{s-1}+(s-1)\Gamma(s-1,z)$ twice, we find that for 
$z$ with $\Re z>0$, we have
\begin{align}\label{UNIQUEPOLELEMPF6}
\eta_a(z)
&=w_1+w_2+w_3
\quad
\text{with }\:
\begin{cases}
w_1=(-iz)^{\frac1a}\Gamma(1-a^{-1})
\\
w_2=a^{-1}(-iz)^{-1}e^{iz}
\\
w_3=-\frac{a+1}{a^2}(-iz)^{\frac1a}\Gamma\bigl(-1-a^{-1},-iz\bigr),
\end{cases}
\end{align}
wherein $(-iz)^{\frac1a}=\exp(\frac1a\log(-iz))$ with the principal branch of the logarithm;
$-\pi<\Im\log(-iz)<0$.

Let $n\in\Z^+$ and $z=x-iy\in\Gamma_n$.
We wish to apply Lemma \ref{SPECINCOMPLGAMMABOUNDLEM} with $s=1+a^{-1}$ and with $-iz$ in place of $z$.
In order to justify this application, we have to check that either
$x\geq\pi$, $y\geq x\geq\frac34\pi$ or $y\geq2x\geq\pi$;
this is clear if $n\geq2$, since then $x>\pi$, %
and if $n=1$, then
the claim follows using \eqref{CURVEDEF}, $\tan(\frac14\pi)=1$ and $\tan(\frac38\pi)>2$.
The conclusion from Lemma \ref{SPECINCOMPLGAMMABOUNDLEM} is that %
$|w_3|<|w_2|$ holds in \eqref{UNIQUEPOLELEMPF6}.   %
We also note that
$\arg\bigl(w_1/w_2\bigr)\in(1+a^{-1})(-\tfrac12\pi+\arg(z))-x+2\pi\Z$,
and by \eqref{CURVEDEF}, %
we have $x\in((2n-\frac32)\pi,(2n-\frac12)\pi]$ and
$\arg(z)=-(n-\frac14)\pi+\frac12x\in(-\frac12\pi,0]$;
together these imply that $\arg\bigl(-ize^{-iz}w_1\bigr)$ lies in
$\bigl[-\frac12a^{-1}\pi,(\frac12-a^{-1})\pi\bigr)\subset(-\frac12\pi,\frac12\pi)$,
i.e.\ that $\Re(w_1/w_2)>0$. 
Moreover, $|w_3|<|w_2|$ forces $\Re((w_2+w_3)/w_2)>0$;
hence we conclude that $\Re((w_1+w_2+w_3)/w_2)>0$, i.e.\ that
\begin{align}\label{UNIQUEPOLELEMPF4}
\Re(-ize^{-iz}\eta_a(z))>0\quad\text{for all }\: z\in\Gamma_n.
\end{align}
This shows that $\eta_a(z)$ has no zeros along $\Gamma_n$,
and also gives a precise control on the variation of $\arg\eta_a(z)$ along $\Gamma_n$.

Next, from \eqref{UNIQUEPOLELEMPF6} and Lemma \ref{INCPGAMMAEASYLEM}, we see that 
for $z=x-iy$ with $y$ large and $x>0$ bounded,  %
we have 
$\eta_a(z)=w_1+w_2+w_3=w_2(1+O(y^{-1}))$,
and thus $\arg\eta_a(z)\in\pi+x+O(y^{-1})+2\pi\Z$.
Also note that $\Re\eta_a(z)>0$ for all $z\geq0$, %
since $\Re\Phi_a(z)>0$ for all $z>0$ (as noted previously) 
and %
$\eta_a(0)=1$.
Using these facts together with \eqref{UNIQUEPOLELEMPF4} (applied both for $n$ and $n+1$),
we conclude that for any $n\in\Z^+$ and any sufficiently large $Y>0$ (depending on both $a$ and $n$), 
$\arg\eta_a(z)$ increases by $2\pi$ as $z$ travels around the boundary of $S_n\cap\{\Im z\geq -Y\}$ in the positive direction.
Hence, by the argument principle, $\eta_a(z)$ has a unique simple zero in the interior of $S_n$.
Using the symmetry $\eta_a(-\overline z)=\overline{\eta_a(z)}$, one proves the same fact also for $S_0$ and
any $S_n$, $n<0$.
This completes the proof of the proposition. 
\end{proof}

From now on, we write $\zeta_n=\zeta_n(a)$ for the unique pole of $z\Psi_a(z)$ in $S_n$ ($n\in\Z$).
By symmetry we have $\zeta_{-n}=-\overline{\zeta_n}$ for all $n$, and in particular $\zeta_0$ lies
on the negative imaginary axis.
Figure \ref{ZEROTRACKINGFIG} below shows the curves traced by $\zeta_0,\ldots,\zeta_4$ as $a$ varies.

The next lemma gives an asymptotic formula for $\zeta_n$ ($n>0$) with an error which is small whenever at least one of
$n$, $a$ and $(a-1)^{-1}$ is large.
\begin{lem}\label{ZETANUNIFASYMPTLEM1}
We have, uniformly over all $a>1$ and all $n\in\Z^+$,
\begin{align}\label{ZETANUNIFASYMPTLEM1RES}
\zeta_n=(2n-a^{-1})\pi+(1+a^{-1})\arctan\Bigl(\frac{2\pi n}{Y_n}\Bigr)
-iY_n+O\biggl(\frac1{n+\log\bigl|\Gamma(-a^{-1})\bigr|}\biggr),
\end{align}
where $Y_n$ equals the unique root $y>0$ of the equation
\begin{align}\label{ZETANUNIFASYMPTLEM2}
y-\sfrac12(1+a^{-1})\log\bigl((2\pi n)^2+y^2\bigr)=\log\bigl|\Gamma(-a^{-1})\bigr|.
\end{align}
\end{lem}
(Regarding the error term in \eqref{ZETANUNIFASYMPTLEM1RES}, we remark that $|\Gamma(-a^{-1})|>3$, 
and thus that $\log|\Gamma(-a^{-1})|>1$, for all $a>1$.)

\begin{proof}
Using \eqref{UNIQUEPOLELEMPF6} and Lemma \ref{INCPGAMMAEASYLEM}, 
together with the fact that $\eta_a(\zeta_n)=0$, we get
\begin{align}\label{ZETANUNIFASYMPTLEM1PF4}
\Gamma(-a^{-1})=(-i\zeta_n)^{-1-\frac1a}e^{i\zeta_n}(1+O(|\zeta_n|^{-1})),
\qquad\forall a>1,\: n\in\Z^+.
\end{align}
Writing $\zeta_n=x_n-iy_n$ ($x_n,y_n>0$) and taking absolute values in \eqref{ZETANUNIFASYMPTLEM1PF4}, we get
\begin{align}\label{ZETANUNIFASYMPTLEM1PF4a}
\bigl|\Gamma(-a^{-1})\bigr|=|\zeta_n|^{-1-\frac1a}e^{y_n}(1+O(|\zeta_n|^{-1})).
\end{align}
Now, using the facts that $|\zeta_n|\geq x_n>(2n-\frac32)\pi\gg n$ 
and $|\Gamma(-a^{-1})|\to\infty$ as $a\to1^+$ or $a\to\infty$,
we conclude that $y_n$ must be large whenever at least one of $n$, $a$ and $(a-1)^{-1}$ is large;
and due to the form of the error term in \eqref{ZETANUNIFASYMPTLEM1RES}, we may without loss of generality
restrict to the case when this holds.
Note that also $|\zeta_n|$ must be large, since $|\zeta_n|\geq y_n$.

In more precise terms, we have, considering the logarithm of equation \eqref{ZETANUNIFASYMPTLEM1PF4a},
\begin{align}\label{ZETANUNIFASYMPTLEM1PF1}
y_n=\tfrac12(1+a^{-1})\log(x_n^2+y_n^2)+\log\bigl|\Gamma(-a^{-1})\bigr|+O(|\zeta_n|^{-1}).
\end{align}
In particular, using $(2n-\frac32)\pi<x_n<(2n+\frac32)\pi$ 
and also $\log(x_n^2+y_n^2)\leq \frac12y_n+2\log n$ (which holds since $y_n$ is large), we conclude that
\begin{align}\label{ZETANUNIFASYMPTLEM1PF12}
y_n\asymp\log n+\log|\Gamma(-a^{-1})|;
\quad\text{and thus }\:
|\zeta_n|\asymp x_n+y_n\asymp n+\log\bigl|\Gamma(-a^{-1})\bigr|.
\end{align}
(Note: ``$\asymp$'' means ``both $\ll$ and $\gg$''.)
Now $x_n^2+y_n^2=((2\pi n)^2+y_n^2)(1+O(|\zeta_n|^{-1}))$,
and thus, in \eqref{ZETANUNIFASYMPTLEM1PF1}, we may replace ``$\log(x_n^2+y_n^2)$'' by
``$\log((2\pi n)^2+y_n^2)$''; the error from this operation is subsumed in the error term $O(|\zeta_n|^{-1})$.
We also note that the expression in the left-hand side of \eqref{ZETANUNIFASYMPTLEM2} is an increasing
function of $y>0$, which is negative for small $y$ and 
the derivative of which lies in the interval $(1-(2\pi)^{-1},1]$, for all $y>0$.
It follows that $Y_n$ (in the statement of the lemma) is well-defined,
and also that
\begin{align}\label{ZETANUNIFASYMPTLEM1PF10}
y_n=Y_n+O\bigl(|\zeta_n|^{-1}\bigr).
\end{align}

Next, taking the argument of both sides of \eqref{ZETANUNIFASYMPTLEM1PF4}, we get
\begin{align}\label{ZETANUNIFASYMPTLEM1PF11}
x_n=(1-a^{-1})\tfrac\pi2+(1+a^{-1})\arg(\zeta_n)+2k\pi+O(|\zeta_n|^{-1})
\quad\text{for some }\:k\in\Z,
\end{align}
where $-\frac\pi2<\arg(\zeta_n)<0$.
Clearly $(2k-1)\pi-O(|\zeta_n|^{-1})<x_n<(2k+\frac12)\pi+O(|\zeta_n|^{-1})$,
and in fact, since $-\arg(\zeta_n)\gg y_n|\zeta_n|^{-1}$ and $y_n$ is large, we even have
$x_n<(2k+\frac12)\pi$.
But also $x_n>(2n-\frac32)\pi$; hence $k\geq n$.
On the other hand, since $\zeta_n$ lies to the left of the curve
$\Gamma_{n+1}$, we have $x_n<2\arg(\zeta_n)+(2n+\frac32)\pi$, and 
using this fact in \eqref{ZETANUNIFASYMPTLEM1PF11}, we get
$(1-a^{-1})\arg(\zeta_n)>(2(k-n)-1-\frac12a^{-1})\pi-O(|\zeta_n|^{-1})$.
This forces $k\leq n$, since $\arg(\zeta_n)<0$ and $|\zeta_n|$ is large.
Hence we have proved that $k=n$.
Finally, using \eqref{ZETANUNIFASYMPTLEM1PF10} and $(2n-\frac32)\pi<x_n<(2n+\frac32)\pi$, we get
$\bigl|\arg(\zeta_n)+\arctan(\frac{Y_n}{2\pi n})\bigr|\ll Y_n|\zeta_n|^{-2}\ll|\zeta_n|^{-1}$.
Now \eqref{ZETANUNIFASYMPTLEM1RES} follows from \eqref{ZETANUNIFASYMPTLEM1PF10}, \eqref{ZETANUNIFASYMPTLEM1PF11}
and \eqref{ZETANUNIFASYMPTLEM1PF12}.
\end{proof}

We may also remark that $Y_n$, as defined in Lemma \ref{ZETANUNIFASYMPTLEM1}, satisfies
\begin{align}\label{YNASYMPT}
Y_n=G+\frac{1+a^{-1}}2\biggl(1+\frac{(1+a^{-1})G}{(2\pi n)^2+G^2}\biggr)\log\bigl((2\pi n)^2+G^2\bigr)
+O\Bigl(\frac1{n+G}\Bigr),
\end{align}
with $G=\log|\Gamma(-a^{-1})|$.
This is proved by direct substitution in \eqref{ZETANUNIFASYMPTLEM2}, %
using the properties of the left-hand side in \eqref{ZETANUNIFASYMPTLEM2} noted in the proof of 
Lemma \ref{ZETANUNIFASYMPTLEM1}.

\vspace{5pt}

We will now change the contour in \eqref{CONTOURCHANGE1}.
Let $a>1$ be given, and fix $r>0$ sufficiently small so that
\eqref{CONTOURCHANGE1} holds.
For $n\in\Z^+$ and $Y>0$, we let 
$z_{n,Y}$ be the unique point where $\Gamma_n$ intersects $\{\Im z=-Y\}$,
and let $C_{n,Y}$ be the contour going from $-\infty$ to $-(2n-\frac12)\pi$ along $\R$,
then along $\Gamma_{-n}:=\{-\overline z\col z\in \Gamma_n\}$ to $-\overline{z_{n,Y}}$,
then along $\{\Im z=-Y\}$ to $z_{n,Y}$, %
further along $\Gamma_n$ to $(2n-\frac12)\pi$, and finally %
along $\R$ to $+\infty$. %
By the residue theorem and Proposition \ref{UNIQUEPOLELEM},
for every $n\in\Z^+$ there is some $Y_0=Y_0(a,n)>0$ such that for $Y>Y_0$, we have
\begin{align}\label{CONTOURCHANGE2}
\frac1{2\pi i}\int_{C_r}\Psi_a(z)\,dz
=\frac1{2\pi i}\int_{C_{n,Y}}\Psi_a(z)\,dz-\sum_{m=1-n}^{n-1}\Res_{z=\zeta_m}\Psi_a(z).
\end{align}
Now let $w_1,w_2,w_3$ be as in \eqref{UNIQUEPOLELEMPF6}.
By Lemma \ref{INCPGAMMAEASYLEM} there is an $N=N(a)\in\Z^+$ such that
$|w_3|\leq\frac12|w_2|$ for all $z\in\Gamma_n$, $n\geq N$.
Using also $\Re(w_1/w_2)>0$ for all $z\in\Gamma_n$, we get
$|\eta_a(z)|=|w_1+w_2+w_3|\geq\frac{\sqrt3}2|w_1|$ %
and thus %
$|\Psi_a(z)|\ll n^{-1-\frac1a}e^{\Im z}$ for all $n\geq N$ and $z\in\Gamma_n$.
Also, for any fixed $a$ and $n$, we have $|\eta_a(z)|\gg Y^{-1}e^Y$ for all $z\in C_{n,Y}\cap\{\Im z=-Y\}$
(cf.\ Lemma \ref{INCPGAMMAEASYLEM} and \eqref{UNIQUEPOLELEMPF6});
thus $|\Psi_a(z)|\ll e^{-2Y}$ for these $z$.
The above bounds imply
$\lim_{n\to\infty}\bigl(\lim_{Y\to\infty}\int_{C_{n,Y}}|\Psi_a(z)|\,|dz|\bigr)=0$,
and so %
\begin{align}\label{DISTREXPLFORMULA}
\PP\bigl(\sigma_{\{T_j\}}>c\bigr)
=\frac1{2\pi i}\int_{C_r}\Psi_a(z)\,dz=-\lim_{n\to\infty}\sum_{m=-n}^{n}\Res_{z=\zeta_m}\Psi_a(z)
=\sum_{n=-\infty}^{\infty} ae^{-2i\zeta_n}.
\end{align}
Here the last equality follows from an easy calculation using \eqref{PHIAYDER} and \eqref{PSIADEF},
noticing that the sum is absolutely convergent,
since, by Lemma \ref{ZETANUNIFASYMPTLEM1} and \eqref{YNASYMPT}, we have
\begin{align}\label{BOUNDFORABSCONV1}
\bigl|ae^{-2i\zeta_n}\bigr|\ll ae^{-2G}(|n|+G)^{-2(1+\frac1a)},\qquad\forall a>1,\: n\in\Z\setminus\{0\}.
\end{align}
One also checks that the formula \eqref{DISTREXPLFORMULA} may be differentiated termwise with respect to $a$,
yielding %
\begin{align}\label{DENSITYEXPLFORMULA}
f(c)=2\sum_{n=-\infty}^\infty e^{-2i\zeta_n}\Bigl(2ai\Bigl(\frac d{da}\zeta_n\Bigr)-1\Bigr)
\end{align}
for the density function (cf.\ \eqref{DENSITY}).

\vspace{5pt}

For $c$ not too large, the formula \eqref{DENSITYEXPLFORMULA} can be used to compute $f(c)$ numerically to a decent precision.
We have implemented this in \cite[numdensity.mpl]{sts}.
Our experiments indicate that for any given $a>1$ ($a=2c$) and $n\in\Z^+$, the asymptotic formula in
Lemma \ref{ZETANUNIFASYMPTLEM1} is sufficiently accurate so that it can be used as the initial value
in the Newton iteration algorithm solving for $\Phi_a(z)=0$, with rapid convergence.
Also, $\frac d{da}\zeta_n$ is computed using
\begin{align*}
\frac d{da}\zeta_n &=a\zeta_n^{1+\frac1a}e^{-i\zeta_n}\Bigl(\frac{\partial}{\partial a}\Phi_a(z)\Bigr)_{|z=\zeta_n}
\\
&=-a^{-2}\zeta_n^{1+\frac1a}e^{-i(\zeta_n+\frac{\pi}{2a})}\biggl(\Gamma'(-a^{-1})-\int_{-i\zeta_n}^\infty
e^{-u}u^{-1-\frac1a}(\log u)\,du\biggr),
\end{align*}
which most often can be evaluated very quickly via repeated integration by parts;
in the remaining cases we use numerical integration.

The data for the graph in Figure \ref{GRAPH} can be found in
\cite[density.dat]{sts}; it was assembled by computing $f(c)$ ($c=\frac12a$)
for $a=1+\frac1{100}k$, $k=1,2,\ldots,400$.
For each $a$-value we truncated the sum in \eqref{DENSITYEXPLFORMULA} at $|n|\leq400$
(using also the obvious $n\leftrightarrow -n$ symmetry).
It turns out that the terms in \eqref{DENSITYEXPLFORMULA} decay roughly as $n^{-2(1+\frac1a)}$ as $n\to\infty$
(cf.\ \eqref{BOUNDFORABSCONV1}).
In particular we have slower convergence for larger $a$ and this is seen in the computations:
Our numerics indicate that we obtain the first few $f(c)$-values to within an absolute error $\lesssim 10^{-11}$,
whereas for $a$ near $5$ (where $f(c)\approx 0.05$) the error is $\lesssim 10^{-6}$.
Of course the precision can be improved by including more terms in \eqref{DENSITYEXPLFORMULA},
again cf.\ \cite[numdensity.mpl]{sts}.

\begin{figure}
\begin{center}
\begin{minipage}{0.8\textwidth}
\unitlength0.1\textwidth
\begin{picture}(13,8)(2,0)
\put(1,8){\includegraphics[width=0.8\textwidth,angle=270]{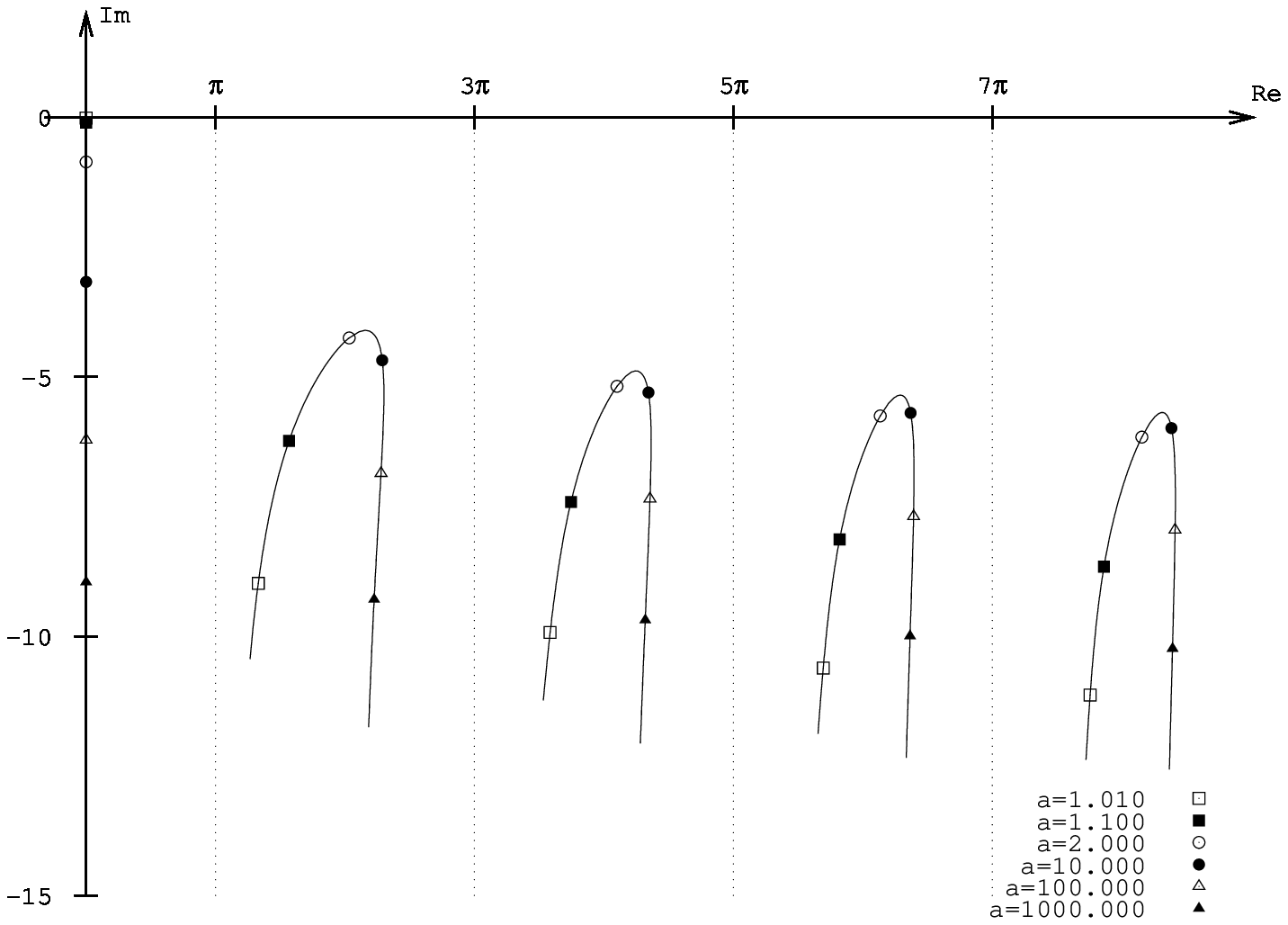}}
\end{picture}
\end{minipage}
\end{center}
\caption{The curves traced by the poles $\zeta_1,\zeta_2,\zeta_3,\zeta_4$ (and $\zeta_0$) as $a$ varies.}\label{ZEROTRACKINGFIG}
\end{figure}

\end{document}